\documentclass[preprint,12pt]{elsarticle}
\usepackage[noend, boxruled,vlined]{algorithm2e}
\usepackage{graphicx}
\usepackage{epsfig}
\usepackage{epstopdf}
\usepackage{amsmath}
\usepackage{placeins}
\usepackage{amsfonts,amsmath, amsthm}
\usepackage{verbatim}
\usepackage{xfrac}

\newtheorem{theorem}{Theorem}[section]
\newtheorem{lemma}{Lemma}[section]
\newtheorem{rem}{Remark}[section]

\DeclareMathOperator{\So}{SO}

\DeclareMathOperator{\area}{area}
\DeclareMathOperator{\trace}{trace}

\journal{Computer Methods in Applied Mechanics and Engineering}

\begin{document}

\begin{frontmatter}



\title{A geometric mesh smoothing algorithm 
       related to damped oscillations}

\author{Dimitris Vartziotis}
\address{NIKI Ltd. Digital Engineering, Research Center, 205 Ethnikis Antistasis Street, 45500 Katsika, Ioannina, Greece}
\author{Doris Bohnet}
\address{TWT GmbH Science \& Innovation, Mathematical Research, Ernsthaldenstr. 17, 70565 Stuttgart, Germany}

\begin{abstract}
We introduce a smoothing algorithm for triangle, quadrilateral, tetrahedral and hexahedral meshes whose centerpiece is a simple geometric triangle transformation. The first part focuses on the mathematical properties of the element transformation. In particular, the transformation gives rise directly to a continuous model given by a system of coupled damped oscillations. Derived from this physical model, adaptive parameters are introduced and their benefits presented. The second part discusses the mesh smoothing algorithm based on the element transformation and its numerical performance on example meshes. 
\end{abstract}


\end{frontmatter}

\pagestyle{myheadings}
\thispagestyle{plain}
\markboth{DIMITRIS VARTZIOTIS AND DORIS BOHNET}{GEOMETRIC MESH SMOOTHING} 




\section{Introduction}\label{s.first}
A modern product development process requires that the parts of the final product are digitally tested in a very early stage of development to study the influence of the novelties on their essential properties and to detect failures betimes. Digital tests are also aimed to replace the expensive testing on test benches. A reliable and fast simulation process is certainly essential for the successful adoption of such a digital product development process. \\
Nearly every simulation in engineering is based on solving partial differential equations with certain numerical solution schemes whose accuracy depends on a good discretization of the object in question. Therefore, every simulation process starts with the preprocessing of the initially given discretization, the so called mesh. This preprocessing should considerably improve the mesh quality where \emph{mesh quality} is a very broad term: a good mesh quality should at least always prevent that the solver fails and guarantee faithful results of the numerical solution schemes (see overview in \cite{Shewchuk2002} and for an error analysis e.g. \cite{BG1986},\cite{BC2010}). One has to stress that the interpretation of mesh quality essentially depends on the physical problem to be simulated (see \cite{F99} for simple examples). Although the separation of mesh quality from the physics is therefore misleading (see e.g. \cite{Du2005} for a counterexample), the most widely used mesh qualities (see overviews in \cite{P1994},\cite{Knupp2001}) are in truth \emph{element} qualities and do not a priori correlate to interpolation errors of the simulation. Usually, they somehow quantify the distortion of each element with regard to a regular element. Nevertheless, their usage can be justified for a wide range of simulations (see \cite[p.5 ff.]{knupp2007}). For this reason, we follow this approach and use the traditional element qualities as a first criterion to assess the mesh smoothing results of our algorithm.\\
In order to improve the mesh quality, one usually relocates nodes (called \textit{smoothing}) and changes the mesh connectivity (called \textit{reconnecting} or modifying the topology). We focus in this article on mesh smoothing techniques where the two main approaches are the \textit{geometric} and \textit{optimization-based approaches}.\\
The geometric approach directly moves the nodes such that the overall mesh quality improves. One classic and widely used example is the Laplacian method where every node is recursively mapped onto the barycenter of its neighboring nodes. While it is usually fast and can be easily run in parallel, it can result in distorted or even inverted mesh elements in non-convex regions so that there exist several, mainly heuristic enhancements of this method (e.g. the isoparametric Laplacian method established by \cite{H76}). In contrast to the geometric, the optimizational approach relies on the computation of the local optimum of an objective function  which expresses the global mesh quality. It is therefore clearly effective as long as the objective function is convex and differentiable. As optimizational schemes one can employ classical methods as the conjugate gradient or newton method in the same way as more recent gradient-free approaches like evolutionary algorithms (see e.g. \cite{YK09}). The main disadvantages of the optimization-based approach are its computational cost and the difficulty to choose an appropriate objective function (see the discussion on mesh quality above and cited references). There exist also combinations of the Laplacian method with an optimizational mesh smoothing in concave regions (see e.g. \cite{Freitag1997}) which try to join the good properties of both.\\

In this article we present a geometric mesh smoothing algorithm for triangle, quadrilateral, tetrahedral and hexahedral meshes which is based on one single simple triangle transformation. Being derived from an element-wise transformation, it is in the spirit of the geometric transformation methods developed in a series of articles (see \cite{VW12}, citations within, and overview in \cite[6.3]{Lo2015}) and mathematically analyzed in \cite{VH14b}. However, the present geometric element transformation was initially motivated by rotational symmetry of triangles and exhibits for this reason distinct mathematical properties like e.g. its provable exponential convergence to a regular triangle and its relation to oscillations which for their part influence its usage as a smoothing algorithm.\\
Endowed with the good computational properties of a geometric method, its effectivity can be mathematically proved for a certain subset of triangle meshes and also established by its correlation to a system of differential equations which model coupled damped springs. Due to this interpretation, one might be reminded of spring-based mesh smoothing methods whose theoretical derivation is however reversed (see discussion in \ref{s.springbased}). The deduction of the algorithm and the presentation of its interesting mathematical properties are the main objective of our article and the topic of Sec.~\ref{s.trans}. But we also give a summary of its numerical smoothing results in Sec.~\ref{s.mesh} which serves as a proof of concept.   

\section{The geometric triangle transformation}\label{s.trans}
We call a triple $\Delta=(x_0,x_1,x_2)$ of non-collinear points in the euclidean plane $\mathbb{R}^2$ or space $\mathbb{R}^3$ a \emph{triangle}. A \emph{triangle transformation} is then a map $\theta: \mathbb{R}^n \rightarrow \mathbb{R}^n$ for $n=(3 \times 2), (3 \times 3)$ which maps a triangle $\Delta$ onto a triangle $\theta(\Delta)$. We call it \emph{geometric} if $\theta$ is compatible with the isometry group actions of the euclidean plane (or space, respectively), that is, $\theta$ commutes with rotations, reflections and translations. Any geometric triangle transformation can be generalized to a triangle mesh transformation by applying it separately to any element of the mesh and then mapping every vertex to the barycenter of its images under the triangle transformations. Another example for a geometric triangle transformation is the transformation which underlies the classical GETMe introduced in \cite{VAGW08}. To guarantee a smoothing effect of such a triangle mesh transformation the geometric triangle transformation must fulfill certain criteria: first of all, $\theta(\theta^{n-1}(\Delta))=\theta^n(\Delta)$ with $n \rightarrow \infty$ should converge to an equilateral triangle (see discussion on element quality above). On the other hand, the convergence should not be too fast because - if applied to a mesh of triangles - this could disturb the global smoothing effect as obviously not any mesh topology admits a mesh of equilateral triangles. The advantage of the geometric element transformation methods is exactly that their basis are transformations which convert step by step any element into a regular one. By this iterative process one can control the rate of the converging of the individual element which is then combined with the averaging step of taking the barycenter. The following observation illustrates this property: If one transforms any element directly into a regular and moves then the nodes onto the barycenter of the vertices of these regular elements, the resulting mesh will probably be not very good and contain inverted elements if the mesh topology does not allow a regular mesh. On the other hand, if one applies a regularizing transformation once or twice to every mesh element, one obtains a better mesh than by applying the classical Laplacian method because every elements are moved towards their best shape. The slow convergence of the element transformation inhibits that the regularizing of each element gets dominant over the constraints of the mesh topology while the individual converging leads to the best possible element quality.\footnote{In fact,  this is not just heuristic but, as on-going research shows, a slightly adapted geometric element transformation method could be seen as a global optimization scheme: one can define an appropriate quality measure which is optimized by the GETMe smoothing algorithm.}\\
In the following we present a geometric triangle transformation which has its offspring in the symmetry group of a triangle: an equilateral triangle is preserved by rotations around its centroid by $\sfrac{2\pi}{3}$ and reflections at its medians (for more on motivational background see \cite{vb16}). One could imitate the rotational group action on an arbitrary triangle by a \emph{pseudo-rotation} of the vertices around the centroid. More precisely, if the triple $(x_0,x_1,x_2)$, $x_i \in \mathbb{R}^2$ for $i\in \mathbb{Z}_3$, defines a planar triangle with centroid $c$, we rotate $x_0-c$ counter-clockwisely around $c$ by the angle $\angle(x_0-c,x_{1}-c)$ until it points into the direction of the vector $x_{1}-c$. We denote this new vertex by $x_{1}^{(1)}$. We proceed analogously with the vertices $x_1$ and $x_2$.  As the inner angles of the triangle face are not necessarily identical, each triangle vertex rotates by a different angle around the barycenter, and consequently, the triangle shape changes and its barycenter $c$ moves. Therefore, we call this transformation of the triangle face a \emph{pseudo-rotation}. If we apply this transformation repeatedly, one observes that the triangle converges to an equilateral one. See Fig.~\ref{fig:triangle_rotated} for an illustration.   
\begin{figure}[htbp]
\includegraphics[width=\textwidth]{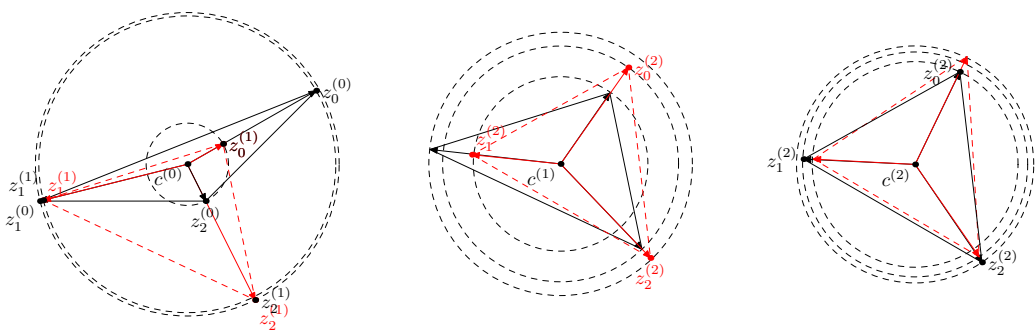}
\caption{The first three iterations of the geometric element transformation represented as rotations: observe how the circle radii approach and the centroid moves to the circumcenter.}
\label{fig:triangle_rotated}
\end{figure}
This simple transformation fulfills the criteria to be a suitable base for a mesh smoothing algorithm, i.e., if applied iteratively, it regularizes an arbitrary triangle, and -- as we see in this article -- has other nice mathematical properties. 
\subsection{Definition of the geometric element transformation}
After these short introductory observations we proceed now with the precise definition of the geometric triangle transformation which is the main subject of this article: 
let $\Delta=(x_0,x_1,x_2) \in (\mathbb{R}^2)^3$ be a planar triangle, indexed in a counter-clockwise fashion, with centroid $c=(\sfrac{1}{3})(x_0+x_1+x_2)$. In the following we denote by $\left\|\;\right\|$ the euclidean metric. Then we define 
\begin{equation}\label{e.firsttransformation}\theta(x_0,x_1,x_2)=\left(x_{0}^{(1)},x_{1}^{(1)},x_{2}^{(1)}\right),\; x_{i}^{(1)}=r_i(x_i-c) + c,\; i=0,1,2,\end{equation}
where $r_i=\left\|x_{i-1}-c\right\|\left\|x_i - c\right\|^{-1}$ for $i \in \mathbb{Z}_3$. Remark that we no longer rotate the triangle nodes as in the informal discussion above where we motivated our transformation. In fact, the vectors $(x_i-c)$ are just rescaled by $r_i$, the ratio of the lengths of $(x_{i-1}-c)$ and $(x_i-c)$. In order to keep the centroid fixed throughout the transformation, we map $x_{i}^{(1)}$ onto $x_{i}^{(1)}-c_{new}+c$, where $i\in\mathbb{Z}_3$ and $c_{new}$ is the centroid of $\theta(\Delta)$. Combining these two transformations, we redefine $\theta$ by  
\begin{align}\label{e.transformation}
\theta(x_0,x_1,x_2) &= \sfrac{1}{3}\begin{pmatrix}2r_{0} & -r_1 & -r_2\\
-r_0 & 2r_1& -r_2 \\
-r_0 & -r_1 & 2r_2\end{pmatrix}\begin{pmatrix}x_{0}-c\\x_1-c\\x_{2}-c\end{pmatrix} + c\\ 
r_i &= \left\|x_{i-1}-c\right\|_2\left\|x_i - c\right\|^{-1},\; i \in \mathbb{Z}_3.\nonumber
\end{align}
\begin{rem}
Let $\Delta=(x_0,x_1,x_2) \in (\mathbb{R}^3)^3$ be a triangle and $E$ the plane spanned by the vectors $(x_1-x_0),(x_2-x_0)$. Then one easily computes that the image $\theta(\Delta)$ lies in $E$, too. Therefore, if one considers a single triangle, it suffices to study a planar triangle.
\end{rem} 
\subsection{Mathematical properties of the geometric element transformation}\label{sec:math}
Let $X$ be the set of non-collinear triples $(x_0,x_1,x_2) \in (\mathbb{R}^2)^3$ which define planar triangles. The geometric triangle transformation $\theta:X \rightarrow X$ is obviously well-defined, i.e. it maps a triangle onto a triangle. Further, one easily observes that $\theta$ commutes with the isometry group action on the plane, that is, it does not matter if one first rotates, reflects or translates a triangle and then applies $\theta$ or conversely, first applies $\theta$. This property is desirable for triangle mesh transformations as the transformation should not depend on the position of the mesh in space. Additionally, $\theta$ is invariant under rescaling of a triangle. Summarizing, one could state the following property:
\begin{lemma}\label{l:similar}
Let $\Delta$ and $\Delta'$ be two similar triangles defined by triples $(x_0,x_1,x_2)$ and $(y_0,y_1,y_2) \in (\mathbb{R}^2)^3$, i.e. there exist $A \in \So(2,\mathbb{R})$, $\sigma \in \mathbb{R}^2$ and $a \in \mathbb{R}^+$ such that $a(Ax_i + \sigma)=y_i$ for $i=0,1,2$. Then the images $\theta(\Delta)$ and $\theta(\Delta')$ are also similar two each other, in particular, it holds that $a(Ax_{i}^{(1)} + \sigma)=y_{i}^{(1)}$ for $i=0,1,2$. 
\end{lemma}
The proof of this lemma is a simple computation, using the fact that rotations and translations are isometries and that any norm is absolutely homogeneous (see \ref{ap:proof}). We use Lemma~\ref{l:similar} to redefine $\theta$ by 
\begin{equation}\label{eq:rescaling}
\theta(\Delta)=\sqrt{\area(\theta(\Delta))^{-1}\area(\Delta)}\theta(\Delta)
\end{equation}
to keep the area of the triangle constant.\\
Further and more importantly, one can prove that $\theta$ given by Eq.~\ref{e.transformation} regularizes any triangle:
\begin{theorem}\label{t.convergence}
Let $\Delta \in X$ be given. Then $\theta^n(\Delta)$ for $n \rightarrow \infty$ converges to an equilateral triangle.  
\end{theorem} 
\begin{proof}
Let $\Delta=(x_0^{(0)},x_1^{(0)},x_2^{(0)})$ be an arbitrary triangle. Denote its iterates by $\theta^n(\Delta)=(x_0^{(n)}, x_1^{(n)},x_2^{(n)})$. The proof is based on the fact that the triangle $\Delta$is equilateral if and only if the distances $\left\|x_i^{(0)}-c\right\|, i=0,1,2$, from the vertices to the centroid are all equal. Recall that the centroid of an arbitrary triangle does not generally coincide with its circumcenter. Proving that the distances to the centroid are equal is equivalent to the fact that the centroid coincides with the circumcenter. Let $r^{(n)}_i$ be the fraction $\left\|x_{i-1}^{(n)}-c\right\|\left\|x_i^{(n)}-c\right\|^{-1}$. One looks at the sequence $\left(r^{(n)}_i\right)_{n\geq 0}$ for $i=0,1,2$. One shows by computation the following:
\begin{enumerate}\label{enum:properties}
\item for any $n\geq 0$, one has $\max_{i=0,1,2}^2\left\|x_i^{(n+1)}-c\right\|< \max_{i=0,1,2}\left\|x_i^{(n)}-c\right\|$, and
\item for any $n \geq 0$, $\min_{i=0,1,2}\left\|x_i^{(n)}-c\right\| < \min_{i=0,1,2}\left\|x_i^{(n+1)}-c\right\|$. 
\end{enumerate}
Further, we define the following upper and minor bounds for $r^{(n)}_i$ for any $n \geq 0$ and $i=0,1,2$ by 
\begin{align}
a_n:=\min_{i=0,1,2}\left\|x_i^{(n)}-c\right\|\left(\max_{i=0,1,2}\left\|x_i^{(n)}-c\right\|\right)^{-1} &\leq r_i^{(n)}\label{inequality_1}\\
b_n:=\max_{i=0,1,2}\left\|x_i^{(n)}-c\right\|\left(\min_{i=0,1,2}\left\|x_i^{(n)}-c\right\|\right)^{-1}&\geq r_i^{(n)}\label{inequality_2}.
\end{align}
The inequalities~\ref{inequality_1} and \ref{inequality_2} imply that $a_n \leq 1$ is a strictly increasing and $b_n \geq 1$ a strictly decreasing sequence, both bounded by $1$. Consequently, we have $\lim a_n = \lim b_n =1$. With the basic Squeeze Theorem, we conclude that $\lim r_i^{(n)} =1$ finishing the proof.
\end{proof}   
Further, one could prove the following stronger result: 
\begin{theorem}\label{t:attractor}
Let $\Delta_{eq} \in X$ be an equilateral triangle and $\Lambda$ the set of triangles similar to it. Then $\Lambda$ is a global attractor for $\theta$ given by Eq.~\ref{e.transformation}. In particular, for any planar triangle $\Delta$ the sequence $\theta^n(\Delta)$ converges uniformly at exponential rate to one point in $\Lambda$ for $n\rightarrow \infty$. 
\end{theorem}
\begin{proof}[Sketch]
The proof of this theorem relies on an analysis of the dynamics of the map $\theta$ which leads to the observation that the eigenvalues of the Jacobian matrix of $\theta$ at an equilateral triangle are solely responsible for the dynamical properties of the map. More precisely, it suffices to prove that the absolute values of all eigenvalues are strictly smaller than one apart from four eigenvalues equal to one which correspond to the four-dimensional tangent space of $\Lambda$. Using Theorem~\ref{t.convergence} this allows then to conclude that $\Lambda$ is a global attractor. A \emph{global attractor} of a discrete dynamical system as generated by the transformation $\theta:X \rightarrow X$ is a compact set $\mathcal{A}\subset X$ such that  $\bigcap_{n \geq 0} \theta^n(X) = \mathcal{A}$, and there is no subset of $\mathcal{A}$ with this properties. That means, every orbit $\theta^n(\Delta)$ for any $\Delta \in X$ eventually converges to the attractor. 
\end{proof}
In \cite{VB2016} we give a detailed mathematical analysis of a geometric tetrahedron transformation which could complement the present section Sec.~\ref{sec:math}.
\begin{rem}
The transformation is not ad hoc generalizable to any $n$-gon, $n \geq 4$: for example, one easily notices that a quadrilateral converges under the transformation above, applied in a strictly analogous fashion, not necessarily to a square, but to a rectangle where the distances from its vertices to its centroid are pairwise equal. Non-convex quadrilaterals pose even more problems. Nevertheless, there is a way to regularize quadrilaterals by defining appropriate triangles inside and then applying the triangle transformation to these inner triangles (see Sec.~\ref{s.meshtransformation}).
\end{rem}
\subsection{Remarks on the generalization to triangle meshes}\label{s:generalization}
The geometric triangle transformation $\theta$ can be utilized to define a \emph{geometric triangle mesh transformation} $\Theta$ in the following way: let $x$ be a node with adjacent triangles $\Delta_0,\dots,\Delta_{k(x)}$, and denote the image of $x$ under the triangle transformations $\theta_i$ of $\Delta_i$ by $x^{i}_{new}$, then we set \begin{equation}\label{eq:Theta}\Theta(x):=\sfrac{1}{k(x)}\sum_{i=0}^{k(x)}x_{new}^i\end{equation} onto the barycenter of its images. 
\begin{rem}
With this definition, one could generalize Theorem~\ref{t:attractor} above to a certain compact subset of planar triangle meshes asserting the effectiveness of the triangle mesh transformation $\Theta$.
\end{rem}
Let $l_i = \left\|x_i - x_{i-1}\right\|, \;i \in \mathbb{Z}_3$ denote the edge lengths of a triangle $\Delta$. We call \emph{distortion} $q(\Delta)$ of a triangle $\Delta$ the ratio of its shortest and longest edge length, that is $q(\Delta)=\min_{i=0,1,2}(l_i)\max_{i=0,1,2}(l_i)^{-1}$ and \emph{orientation} the sign of its normal vector $(x_1-x_0) \times (x_2-x_0)$. 
\begin{figure}[htbp]
\centering
\begin{minipage}{0.47\textwidth}
\centering
\includegraphics[width=0.6\textwidth]{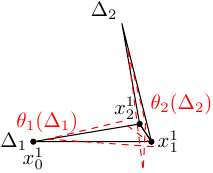}
\end{minipage}
\begin{minipage}{0.47\textwidth}
\centering
\includegraphics[width=0.6\textwidth]{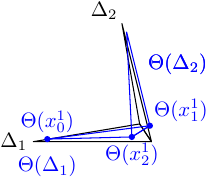}
\end{minipage}
\caption{Triangle flip: mesh of two triangles $\Delta_1,\Delta_2$ where the mesh transformation $\Theta$ reverses the orientation of one of the triangles.}
\label{fig:flips}
\end{figure} 
One notices that the transformation $\Theta$ defined by \ref{eq:Theta} is not orientation-preserving on the whole set of planar triangle meshes: for example, if two distorted triangles share one edge, the orientation of one of the triangles could be reversed  (see Fig.~\ref{fig:flips}). For that reason, one defines a compact subset of meshes whose element distortion is bounded from below and proves that the mesh transformation $\Theta$ is orientation-preserving on this subset. However, the defined bound is not sharp and difficult to calculate precisely as it is not the distortion of a single triangle which counts (in fact, the triangle transformation regularizes any single triangle), but the combination of distortions of at least two triangles and their specific position to each other. More precisely, let $x_0,x_1,x_2,x_3$ define a mesh of two triangles $\Delta_0,\Delta_1$ with centroids $c_0,c_1$ as in Fig.~\ref{fig:flips}. Assume for simplicity that $\left\|x_1-c_0\right\|=\left\|x_2-c_0\right\|<\left\|x_0-c_0\right\|$ and $\left\|x_1-c_1\right\|=\left\|x_2-c_1\right\| < \left\|x_3-c_1\right\|$. Therefore, the distortions are $q(\Delta_0)=\left\|x_1-c_0\right\|\left\|x_1-c_1\right\|^{-1}$ and $q(\Delta_1)=\left\|x_1-c_1\right\|\left\|x_3-c_1\right\|^{-1}$. One computes then 
\begin{align*}
\Theta(x_1)&=\sfrac{1}{2}(\theta_0(x_1) + \theta_1(x_1))\\
&=\sfrac{1}{2}(q(\Delta_0)^{-1}(x_1-c_0)+c_0+x_1)
\end{align*}
In the same way, one obtains $\Theta(x_2)=\sfrac{1}{2}(x_2 + q(\Delta_1)^{-1}(x_2-c_1)+c_1)$ and $\Theta(x_0)=q(\Delta_0)(x_0-c_0)+c_0$. The vector $x_2$ is therefore principally moved into the direction of $(x_2-c_1)$ scaled by $q(\Delta_2)^{-1}$ while $x_1$ is pushed into the direction of $(x_1-c_0)$ scaled by $q(\Delta_0)^{-1}$. Whether the sign of $(\Theta(x_1)-\Theta(x_0))\times(\Theta(x_2)-\Theta(x_0))$ gets reversed in comparison to the sign of $(x_1-x_0)\times(x_2-x_0)$ depends on the angle of these two vectors to each other and on the distortions $q(\Delta_0),q(\Delta_1)$. In a usual mesh, where one vertex has usually more than two adjacent triangles, the dependencies of distortions and angles become complex as actually the distortion of any triangle element influences the movement of one vertex $x$, but with decreasing weight for increasing distance to $x$. Nevertheless, if one considers a mesh which can be smoothed to a mesh of equilateral triangles one can deduce -- essentially from the continuity of the mesh transformation $\Theta$ -- that there exists a bound $0 < \delta$ such that every mesh whose elements have a distortion $> \delta$ converges to the equilateral mesh.  But as this bound cannot be in general exactly computed, we either use in Sec.~\ref{s.mesh} adaptive parameters or an explicit orientation check to prevent the creation of invalid elements under the mesh transformation. \\
Back to the mathematical properties, one has to acknowledge that the necessary computations to assure the effectiveness of the triangle mesh transformation involve the analysis of huge Jacobian matrices, and the limited knowledge of the exact nature of these matrices restricts the cases where we can really prove the global convergence of this mesh transformation. For that reason, we do not further explore these quite technical mathematical aspects here, but we would rather like to describe how to correlate this discrete transformation to a system of linear differential equations. This is the objective of the following section:     
\subsection{Model of a system of damped oscillations}\label{sec:model}
Looking for a mathematically simpler, but adequate model to handle the transformation $\theta$, we describe the dynamics of $\theta$ by differential equations and try to bypass the computational difficulties in this way. The following observations together with Sec.~\ref{sec:math} provide also an evidence for the numerical results in Sec.~\ref{s.mesh}.
\subsubsection{Derivation and solution of a system of differential equations}
One observes that the quantities, responsible for the convergence properties of $\theta$, are the distances $R_i=\left\|x_i-c\right\|$ of the vertices $x_i$, $i=0,1,2$, to the centroid $c$. Consequently, we consider these as time-dependent variables $R_i(t)$ and describe their dynamics by differential equations. First of all, the fixed points of $\theta$ are exactly the equilateral triangles, so we have to assure that $\dot{R}_i(t)=0$ for $i=0,1,2$ if and only if the distances $R_0(t)=R_1(t)=R_2(t)$ are equal and therefore constant. Further, we observe that the distance $R_0(t)$ increases if $R_2(t)$ is greater than the average distance and decreases otherwise, so we can set 
$$\dot{R}_0(t)=R_2(t) - (\sfrac{1}{3})(R_0(t)+R_1(t)+R_2(t)) = (\sfrac{2}{3})R_2(t) -(\sfrac{1}{3})(R_0(t) + R_1(t)).$$
In this way, we get for the distance vector $R(t)=(R_0(t),R_1(t),R_2(t))$ the following system of linear differential equations: 
\begin{equation}\label{e.system_ode}\begin{pmatrix}\dot{R_0}(t)\\\dot{R_1}(t)\\\dot{R_2}(t)\end{pmatrix}=\begin{pmatrix}-\sfrac{1}{3}& - \sfrac{1}{3} & \sfrac{2}{3}\\\sfrac{2}{3} & -\sfrac{1}{3} & -\sfrac{1}{3}\\-\sfrac{1}{3}& \sfrac{2}{3} & -\sfrac{1}{3}\end{pmatrix}\begin{pmatrix}R_0(t)\\R_1(t)\\R_2(t)\end{pmatrix}\quad\mbox{with} \begin{pmatrix}R_0(0)\\R_1(0)\\R_2(0)\end{pmatrix}=\begin{pmatrix}R_0\\R_1\\R_2\end{pmatrix}.\end{equation}  
One easily computes that this system has as stationary solutions the constant vectors $(a,a,a) \in \mathbb{R}^3, a \in \mathbb{R}^+$ which correspond to the equilateral triangles. 
\begin{rem}
We would like to stress that the presented system of differential equations is a model for the geometric transformation. It is not the case that the transformation is the discretization of the continuous system, e.g. derived from the Euler method. 
\end{rem}
Let us now shortly discuss the dynamics of this system of differential equations: 
\subsubsection{Dynamics of the continuous solutions}
Solving a system of linear differential equations is a standard technique: we compute the eigenvalues of the coefficient matrix as $\lambda_0=0$ and a pair of complex conjugate eigenvalues $\lambda_{1,2}=-\sfrac{1}{2} \pm \sfrac{\sqrt{3}i}{2}$ with corresponding eigenvectors $v_0=(1,1,1)$, and $v_{1,2}=(-\sfrac{1}{2} \pm \sfrac{\sqrt{3}i}{2}, 1, -\sfrac{1}{2} \mp \sfrac{\sqrt{3}i}{2})$. Denoting the initial values by $(R_0,R_1, R_2)$ we obtain as general solution for $R_i(t)$, $i\in\mathbb{Z}_3$,
\begin{equation}\label{e.solution_ode}
R_i(t)=(\sfrac{1}{3})(R_0+R_1+R_2) + (\sfrac{1}{3})\exp(-\sfrac{t}{2})(c_i\cos(\sfrac{\sqrt{3}t}{2}) + s_i\sin(\sfrac{\sqrt{3}t}{2})),
\end{equation}
using the abbreviations $c_i=2R_i - R_{i+1}-R_{i+2}$ and $s_i=\sqrt{3}(R_{i+2}-R_{i+1})$.\\
As the distances between vertices and centroid for any triangle are positive, we get a half line $\left\{(a,a,a) \;\big|\; a\in \mathbb{R}^+\right\}$ of constant solutions as equilibria which correspond to the equilateral triangles. Set $\overline{R}=(\sfrac{1}{3})(R_0+R_1+R_2)$ as the average distance of the initial triangle. As cosine and sine are bounded by $1$, we have for $i=0,1,2$
$$\left|R_i(t) - \overline{R}\right| \leq (\sfrac{1}{3})(\left|c_i\right|+\left|s_i\right|)\exp(-\sfrac{t}{2}) \longrightarrow 0\quad\mbox{for}\; t \rightarrow +\infty.$$
Consequently, every solution $R(t)$ converges at exponential rate to the constant solution $(\overline{R},\overline{R},\overline{R})$. The dynamics could be easily deduced from Fig.~\ref{fig:comparison} (left). 
One could fill in the solution $R(t)$, replacing $R_0, R_1, R_2$, into the transformation $\theta$ given by Equation~\ref{e.transformation} to convert it into a continuous process by $x_i(t)$. This allows the observation that every vertex spirals slowly around the corresponding vertex of the limit triangle while approaching it. If one compares the limit triangle of this continuous process with the one obtained by $\theta$ (e.g. in Fig.~\ref{fig:triangle_rotated}), one observes that they coincide which affirms that our model is adequate. 
\subsubsection{Comparison between continuous solution and discrete transformation}
For further validation of our continuous model we compare the results by numerical computations: let $\Delta$ be a planar triangle with distances $R_0,R_1,R_2$ from its centroid to its vertices. Choosing an appropriate discretization for $t$, we plot the continuous curves $R(t)$ given by the solution of Equation~(\ref{e.system_ode}) against the discrete iterates $R_i^n$ for $n \in \mathbb{N}$ which denote the distances of the $n$th iterate $\theta^n(\Delta)$ of the initial triangle $\Delta$. Fig.~\ref{fig:comparison} shows the accuracy of the continuous description for the discrete transformation and the rapid convergence of the transformation to a regular triangle. \\
\vspace{0.1cm}
\begin{figure}[htbp]
\centering
\begin{minipage}[t]{0.48\textwidth}
\includegraphics[width=\textwidth]{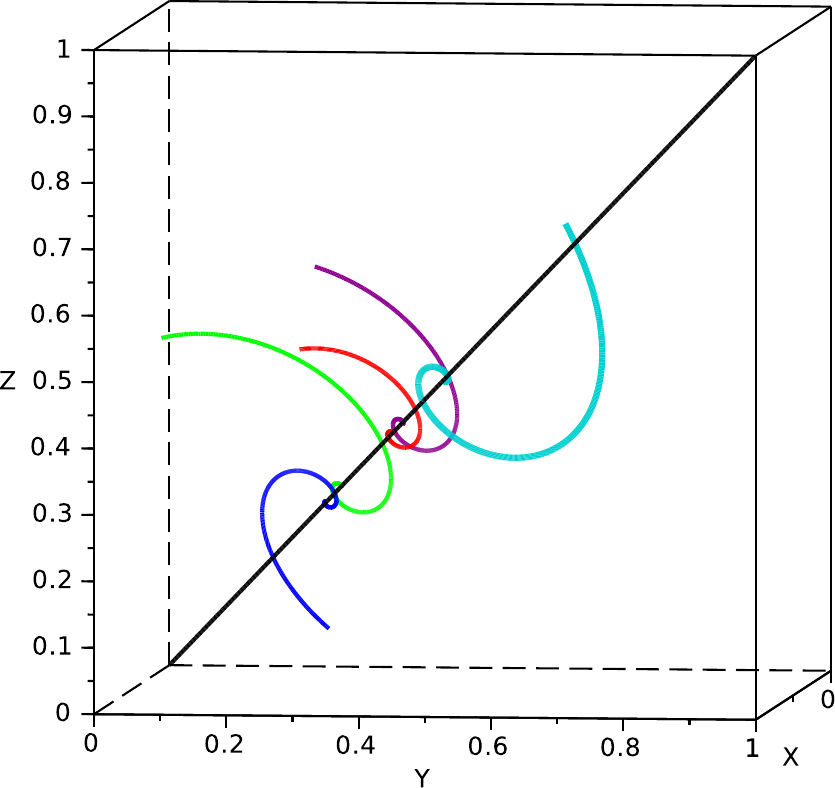}
\end{minipage}
\hfill
\begin{minipage}[t]{0.48\textwidth}
\includegraphics[width=\textwidth]{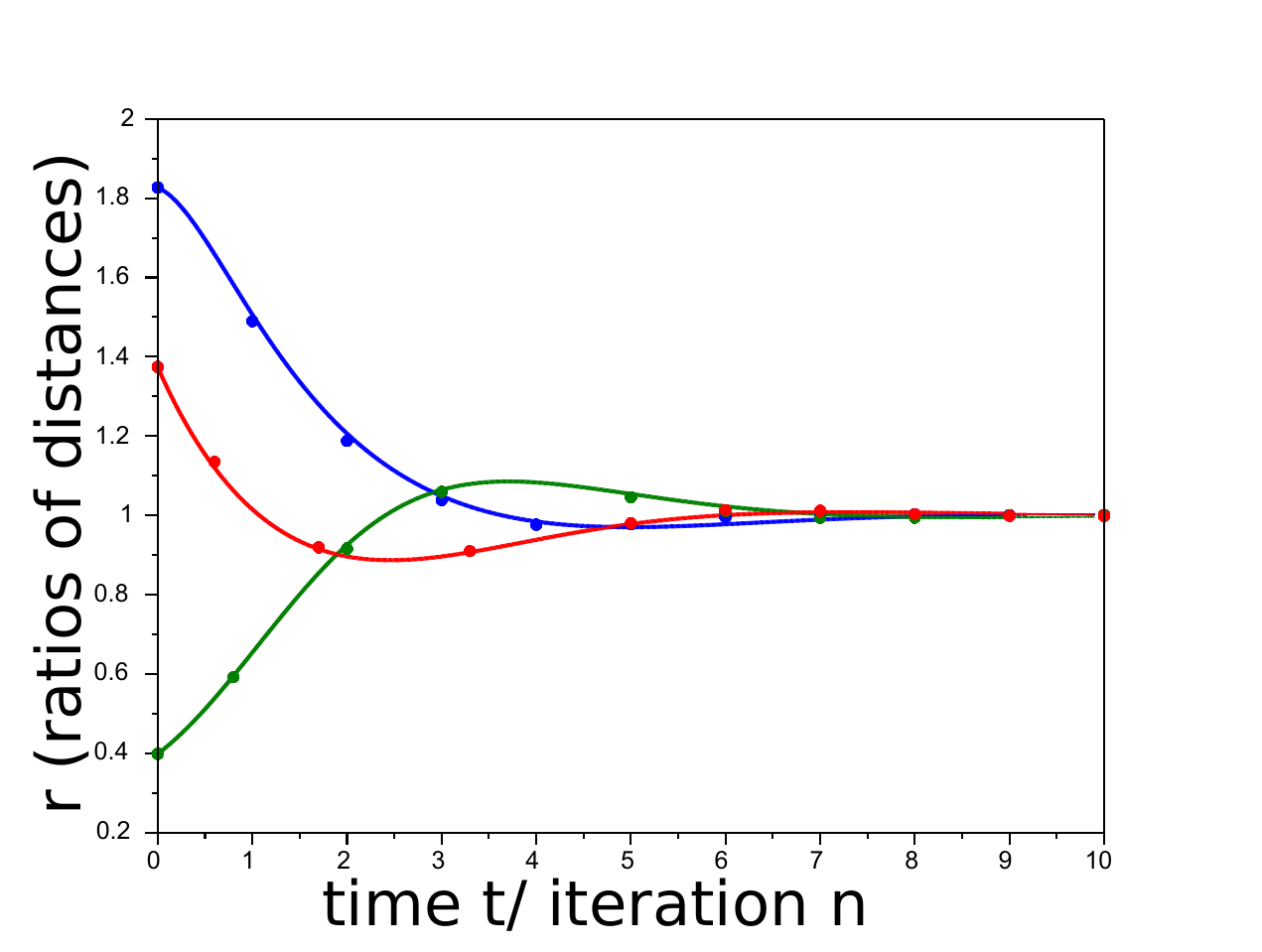}
\end{minipage}
\caption{On the left, the solution curves $R(t)$ for five arbitrarily chosen initial values. Observe how they all approach the stationary solution in spirals. On the right, the continuous curves $r_i(t)$ and discrete iterates $r_i^{(j)}$ for $j=0,\dots,10$ and $i=0,1,2$ for a randomly generated triangle. } 
\label{fig:comparison}
\end{figure} 
\subsubsection{Remarks on the relation to damped oscillations}
Besides from these nice dynamical properties, the system of linear differential equations~(\ref{e.system_ode}) above provides a link to a physical model: 
looking at the solution~(\ref{e.solution_ode}) we observe that 
\begin{align}
\dot{R}(t)&=(\sfrac{1}{3})\exp{(-\sfrac{t}{2})}\begin{pmatrix}(2R_2 - R_1 -R_0)\cos(\sfrac{\sqrt{3}t}{2}) + \sqrt{3}(R_1 - R_0)\sin(\sfrac{\sqrt{3}t}{2})\\(2R_0 - R_1 -R_2)\cos(\sfrac{\sqrt{3}t}{2}) + \sqrt{3}(R_2 - R_1)\sin(\sfrac{\sqrt{3}t}{2})\\(2R_1 - R_2 -R_0)\cos(\sfrac{\sqrt{3}t}{2}) + \sqrt{3}(R_0 - R_2)\sin(\sfrac{\sqrt{3}t}{2})\end{pmatrix}\nonumber\\
&=\begin{pmatrix} R_2(t) - a\\R_0(t) - a\\R_1(t)-a\end{pmatrix} \label{eq:derivative}.
\end{align}
In particular, we get
$$\dot{R}(0)= \begin{pmatrix} R_2- a\\R_0 - a\\R_1-a\end{pmatrix}=(\sfrac{1}{3})\begin{pmatrix}2R_2 -R_0- R_1\\2R_0 - R_1 - R_2\\2R_1 - R_0 - R_2\end{pmatrix}.$$  
Accordingly, using Eq.~\ref{eq:derivative} one has for the second derivatives
$$\ddot{R}(t) = \begin{pmatrix} \dot{R}_2(t)\\\dot{R}_0(t)\\\dot{R}_1(t)\end{pmatrix}=\begin{pmatrix} {R}_1(t)-a\\{R}_2(t)-a\\{R}_0(t)-a\end{pmatrix}.$$
Therefore, the system of linear differential equations could be equivalently written as three linear differential equations of second order where \newline$(R_0,R_1,R_2)=(R_0(0),R_1(0),R_2(0))$ are the initial values: 
\begin{align*}
& \ddot{R}_i(t) + \dot{R}_i(t) + R_i(t) = (\sfrac{1}{3})(R_i + R_{i+1} + R_{i+2}), \\
&R_i(0)=R_i, \; \dot{R}_i(0)=\sfrac{2}{3}R_{i+2} - \sfrac{1}{3}R_{i+1} - \sfrac{1}{3}R_{i},  i \in \mathbb{Z}_3.
\end{align*}
Each of these equations can be seen as describing a damped oscillation. 
Consequently, the dynamics of each distance $R_0(t),R_1(t)$ and $R_2(t)$ can be understood as the dynamics of a damped oscillation which depend on the initial values of each other. 

\subsubsection{Controlling the convergence rate by adaptive parameters}\label{sec:adaptive}
The illustration of the transformation as a damped oscillation helps to understand how one can control the speed of the convergence for each triangle to achieve a better global smoothing effect. Further, a similar approach -- although not arisen from a mathematical model -- was undertaken in \cite{VP13} to improve a geometric smoothing algorithm with success. For this purpose let us introduce three parameters $\alpha_i >0, i=0,1,2,$ which allow us to change proportionally the ratios $r_i, i=0,1,2$. Instead of mapping $x_{i,new}=r_ix_i$ for  $i=0,1,2$ we factorize by the parameter $\alpha_i$ getting 
$x_{i,new}=\alpha_i r_ix_i$ for $i=0,1,2$.
If we keep the centroid in the origin during the transformation we obtain for $i \in \mathbb{Z}_3$
$$x_{i,new}=(\sfrac{2}{3})\alpha_i r_ix_i-(\sfrac{1}{3})\alpha_{i+1}r_{i+1}x_{i+1}-(\sfrac{1}{3})\alpha_{i+2}r_{i+2}x_{i+2}.$$
As the equilateral triangle should be the fixed point of this transformation we impose that $-\sfrac{\alpha_1}{3}-\sfrac{\alpha_2}{3}+\sfrac{2\alpha_0}{3}=0$. This equation allows us to reduce the number of parameters by setting $\alpha_2=2\alpha_0-\alpha_1$. 
Using this observation we change the differential equations for $R(t)$ adequately and compute the solution depending on $\alpha_0$ and $\alpha_1$: the eigenvalues of the system matrix $A(\alpha_0,\alpha_1)$ are $\lambda_0=0$ corresponding to the equilibrium solution and 
$$\lambda_{1,2}=-\sfrac{1}{6}\left(2\alpha_0\alpha_1 -\alpha_1^2+2\alpha_0^2 \pm \sqrt{3}i\sqrt{5\alpha_0^2-4\alpha_1^2+8\alpha_0\alpha_1}\right).$$
Therefore, the basic solutions are the equilibrium $\overline{R}=\sfrac{1}{3}(R_0+R_1+R_2)$ and $y_1(t)=e^{\lambda_1t}v_1$ and $y_2(t)=e^{\lambda_2t}v_2$ where $v_1,v_2$ are the (probably complex) eigenvectors to $\lambda_1,\lambda_2$. 
Consequently, the convergence rate of the solutions depends exclusively on the real part $\mathfrak{R}(\lambda_{1,2}(\alpha_0,\alpha_1))=-\sfrac{1}{6}(2\alpha_0\alpha_1 -\alpha_1^2 + 2\alpha_0^2)$.
\begin{figure}[htbp]
\includegraphics[width=\textwidth]{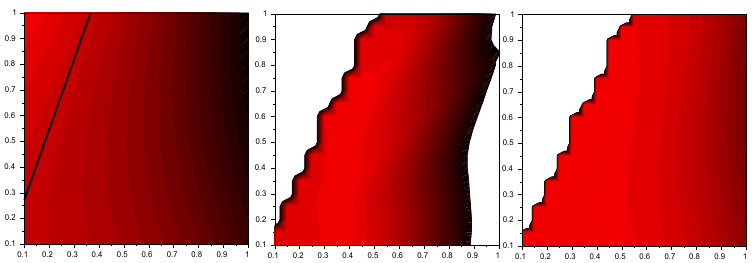}
\caption{Dependence on adaptive parameters: on the right, the contour lines for $(\alpha_0,\alpha_1)\mapsto \mathfrak{R}(\lambda_{1,2}(\alpha_0,\alpha_1))$ are depicted where the line of zeros $\alpha_1=(1+\sqrt{3})\alpha_0)$ is black. The values $\alpha_1<(1+\sqrt{3})\alpha_0)$ corresponds to a negative real part. On the right, we show the influence of the parameters on the mean mesh quality $q(\alpha_0,\alpha_1)$, in the middle the minimal element quality for the unit square triangulation in Fig.~\ref{fig:square}. The mesh transformation as defined in Sec.~\ref{s.mesh} was once applied. The darker the red, the worse the quality. The best minimal element quality $q_{\text{min}}=0.265$ and the best mean quality $q_{\text{mean}}=0.624$ are attained at $(0.4,0.5)$ and $(0.1,0.15)$, respectively. The worst minimal quality $q_{\text{min}}=0.044$ and the worst mean quality $q_{\text{mean}}=0.554$ are attained at $(1,0.1)$ and $(1,0.65)$. Only values with $\alpha_1 < 2\alpha_0$ were considered. These values do not necessarily coincide with the mathematical analysis as the mesh transformation involves that the barycenter of the images of the triangle transformation for each vertex is taken. }
\label{fig:adaptive}
\end{figure} 
If $\mathfrak{R}(\lambda)$ is strictly negative, the solutions $y_{1,2}(t)$ converge uniformly to the equilibrium. For this reason, we study the function $(\alpha_0,\alpha_1)\mapsto \mathfrak{R}(\lambda(\alpha_0,\alpha_1))$: it is helpful to write
$$-\sfrac{1}{3}\alpha_0^2+\sfrac{1}{6}\alpha_1^2-\sfrac{1}{3}\alpha_0\alpha_1 = (\alpha_0\alpha_1)\begin{pmatrix}-\sfrac{1}{3} & -\sfrac{1}{6}\\-\sfrac{1}{6} & \sfrac{1}{6}\end{pmatrix}\begin{pmatrix}\alpha_0\\\alpha_1\end{pmatrix}$$
as a symmetric bilinear form which is indefinite with eigenvalues $\sfrac{1}{12}(-1\pm\sqrt{13})$. The set of zeros is the union of two lines at 0. Supposing $\alpha_0,\alpha_1 > 0$, the half line $\alpha_1(\alpha_0)=(1+\sqrt{3})\alpha_0$ is the only set of zeros. For $\alpha_1 < (1+\sqrt{3})\alpha_0$, the real part is negative and hence, the solutions are uniformly converging to the equilibrium solution (see Fig.~\ref{fig:adaptive}). 
This allows us to conclude how 
$\alpha_0$ and $\alpha_1$ have to be chosen to control the convergence towards an equilateral triangle. Sometimes, it is preferable to decelerate a bit in order to improve the global convergence of the triangle mesh: if one thinks of Fig.~\ref{fig:flips}, the reversion of orientation in this example is mostly due to the factor $q(\Delta_0)^{-1}$,$q(\Delta_1)^{-1}$, the inverse of element qualities, by which the vectors $(x_1-c_0)$,$(x_2-c_1)$ are scaled. The parameter $\alpha_0$ is multiplied to these factors and can therefore impede the orientation-reversion if chosen sufficiently small (see Sec.~\ref{s.mesh} for numerical results).\\
 
\subsection{Demarcation from the spring-based mesh smoothing method}\label{s.springbased}
\subsubsection{Short notes on the classical spring-based mesh smoothing method}
Talking about springs and the equilibrium of their oscillations, one might be reminded of the spring-based mesh smoothing method which is widely used for dynamic mesh smoothing and implemented in common thermodynamical simulation tools (e.g. ANSYS Fluent). 
There is a variety of smoothing methods based on a spring analogy which could differ significantly with respect to results and employed numerical methods. This section is aimed at explaining the differences between our presented method and a classical spring-based mesh smoothing method as implemented in ANSYS Fluent (see User's Guide~\cite{Fluent}): 
the spring-based mesh smoothing method relies on the idea that the edges of a triangle mesh could be interpreted as springs clamped between the vertices. One tries to relocate the vertices such that an equilibrium for the spring force at each vertex of the network is attained for the global spring force of the network composed of the spring force at each vertex: let $x_i \in \mathbb{R}^3$ for $i=0,\dots,N-1$ denote the vertices of the triangle mesh. The force $F_i$ acting on a vertex $x_i$ is by Hooke's law proportional to the displacement of each vertex from the equilibrium state. Consequently, one obtains as force acting on vertex $x_i$ the sum of the forces applied by the springs attached at $x_i$, that is
$$F_i=\sum_{j \in N(i)}k_{ij}(\Delta x_j -\Delta x_i)$$
where $N(i)$ is the set of indices of neighboring vertices of $x_i$, $\Delta x_i$ the displacement of $x_i$ from its equilibrium state and $k_{ij}$ the spring constant of the spring between $x_i$ and $x_j$. Typically, the boundary nodes are fixed, and in order to propagate the constraints of the boundary (see \cite{ZE2005}), one supposes that each spring in equilibrium has length zero. Further, the spring constant is set proportional to the inverse of the edge length (see analysis in \cite{B2000}) to prevent the collision of nodes.As the boundary nodes are fixed, the initial displacement $\Delta x_i^{(0)}$ of the boundary nodes $x_i$ are known. The values for the internal nodes can then be computed iteratively by setting $F_i=0$:
$$\Delta x_i^{(m)} = \frac{\sum_{j \in N(i)}k_{ij}\Delta x_j^{(m-1)}}{\sum_{j \in N(i)}k_{ij}}$$
where $k_{ij}=\sfrac{1}{\left\|x_i-x_j\right\|}$ and $m >0$. One aborts the iteration if $\|\Delta x_{i}^{(m)}-\Delta x_i^{(m-1)}\| < \epsilon$ and updates the vertices for $n > 0$ by $$x_i^{(n)} = x_i^{(n-1)} + \Delta x_i^{(*)}.$$
By introducing different relaxation and damping parameters one could control the influence for boundary vertices to the overall mesh. The solution of the arising systems of linear equations can be achieved with different numerical solution schemes: often the Jacobi method is employed (see \cite{SV2014} for derivation and discussion).
\subsubsection{Differences to the presented geometric smoothing algorithm}
First of all, the two smoothing methods pursue two different approaches: the geometric element transformation relocates every vertex of a triangle (and mesh) explicitly while the spring-based mesh smoothing method is a global approach which is aimed at determining a global equilibrium state of the mesh which amounts consequently to solving a large system of equations. The employed solver are usually iterative what leads to an iterative relocation of the vertices; nevertheless, the approach targets to compute in one step the final vertex positions corresponding to the equilibrium state of the mesh. At the same time, the strategy how to obtain the equilibrium varies a lot in the literature: on the one hand, the Laplacian method is interpreted as spring analogy; on the other hand, one computes the potential energy of the spring network which is then minimized using an optimization method (e.g. \cite{SZ2000}). \\
With regard to the model equations we deduced in Sec.~\ref{sec:model}, the modeled springs of a single triangle are attached to the fixed centroid and influence each other by damping. The main difference to the spring analogy approach is that the discrete geometric triangle transformation models the behavior of damped springs which attain an equilibrium because of the damping while the spring-based method uses the description by springs to compute directly the equilibrium state.

\section{The mesh smoothing algorithm based on the geometric element transformation}\label{s.mesh}
\subsection{Mesh transformation for triangle and tetrahedral meshes}\label{s.meshtransformation}

As already described at the beginning and in Sec.~\ref{s:generalization}, any triangle transformation can be used to define a triangle mesh transformation: firstly, apply the triangle transformation to every triangle element of the mesh, secondly, map the vertex onto the barycenter of its images under the triangle transformations. Let us be more precise:   
let $M_{\Delta}$ be a triangle mesh defined by the set $V=(x_0,\dots,x_{N-1})$, $x_i \in \mathbb{R}^2$ or $x_i \in \mathbb{R}^3$, of vertices, indexed in a counter-clockwise fashion, and the set $E=(\Delta_0, \Delta_1, \dots, \Delta_{n-1})$ of triangle elements $\Delta_i=(i_0,i_1,i_2)$ with $i_j \in \left\{0,\dots,N-1\right\}$. We define the triangle mesh algorithm in Algorithm~\ref{alg:TriangleMesh}.\\
\begin{algorithm}[h]\label{alg:TriangleMesh}
 \caption{Triangle mesh algorithm}
 \KwData{nodes, elements,boundary, NumberofInnerIterations, NumberofIterations, ErrorBound}
 compute mesh quality $\text{Quality}(0)$, set $\text{Quality}(1)=1$\;
 \For{$i=0:\text{NumberofNodes}-1$}
{find triangles $\Delta$ which contain the index $i$, denote the set of indices of adjacent triangles by $J(i)$.}
\While{Quality$(n)$ -Quality$(n-1)$ $>$ ErrorBound \& $n < \text{NumberofIterations}$}{
\For{$i=0:\text{NumberofNodes}-1$}
{update vertex $x_i$ as following:\;
 \For{$k\in J(i)$}
{iterate the triangle $\Delta_k$ formed by $(x_{k_0},x_{k_1},x_{k_2})$ NumberofInnerIteration-times with $\theta_k$ given by (\ref{e.transformation}) to $\Delta_{k,new}$\;
rescale $\Delta_{k,new}$ following Eq.~(\ref{eq:rescaling}) to keep the area constant.\;
\If{orientation of $\Delta_{k,new}$ is reversed,}
{reset $\Delta_{k,new}=\Delta_k$.}}
{update the vertex $x_i$ by the arithmetic mean $\frac{1}{|J(i)|}\sum_{k \in J(i)} \theta_{k}(x_i)$.
  }}
 Apply boundary constraints.}
\end{algorithm} 
\begin{rem}
\begin{enumerate}
\item It has been proven to be more efficient to iterate each element three times before taking the barycenters and updating the mesh vertices.
\item With the abbreviatory expression \emph{Apply boundary constraints} we summarize different method to conserve the shape of the meshed object: vertices on the boundary of a planar mesh are projected onto the boundary after the iteration. Vertices on the surface or on feature lines of volume meshes are fixed. 
\end{enumerate}
\end{rem}
One can also adapt the algorithm in order to smooth quadrilateral meshes. We define this smoothing algorithm in Algorithm~\ref{alg:quadmesh}.\\
\begin{figure}
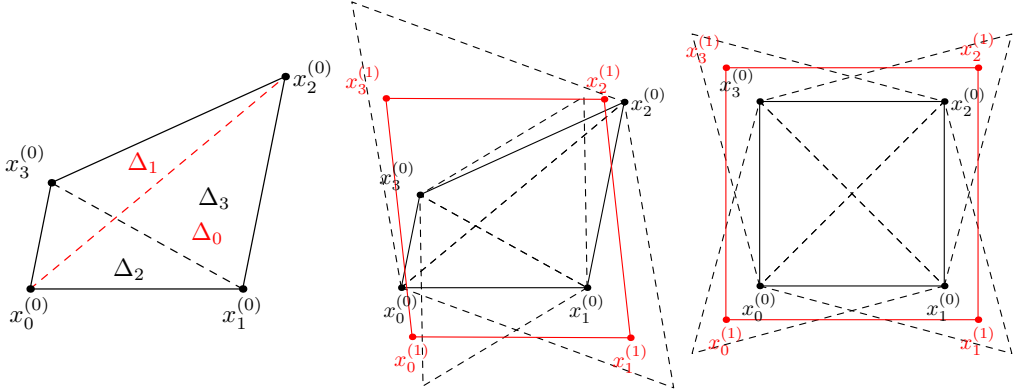

\begin{minipage}{0.32\textwidth}
\includegraphics[width=\textwidth]{quad.mps}
\end{minipage}
\begin{minipage}{0.32\textwidth}
\includegraphics[width=\textwidth]{quad_iterated.mps}
\end{minipage}
\begin{minipage}{0.32\textwidth}
\includegraphics[width=\textwidth]{quadrat.mps}
\end{minipage}
\caption{On the left, a quadrilateral with the four triangles $\Delta_0,\dots,\Delta_3$ which are created by its two diagonals and in the middle its first iterate of the quadrilateral (keeping the diagonals fixed). On the right, we show the first iterate of a square (without rescaling the area) to illustrate that the transformation keeps squares invariant (up to scaling). One iterates a non-convex quadrilateral in the exactly same way. }
\label{fig:quad}
\end{figure}
\begin{algorithm}[h]
\label{alg:quadmesh}
\caption{Quadrilateral mesh algorithm}
 \KwData{nodes, elements,boundary, NumberofInnerIterations, NumberofIterations, ErrorBound}
 compute mesh quality $\text{Quality}(0)$, set $\text{Quality}(1)=1$\;
 \For{$i=0:\text{NumberofNodes}-1$}
{find quadrilateral $Q$ which contain $x_i$. Denote the set of indices of adjacent quadrilaterals by $J(i)$.}
\While{Quality$(n)$ -Quality$(n-1)$ $>$ ErrorBound \& $n < \text{NumberofIterations}$}{ 
\For{$i=0:\text{NumberofNodes}-1$}
{update vertex $x_i$ as following:
 \For{$k\in J(i)$}
{define four triangles with the help of the diagonals (see Fig.~\ref{fig:quad}) and apply the Triangle Mesh Algorithm~\ref{alg:TriangleMesh} to each of them a NumberofInnerIterations-times,\;set 
$\overline{x}_{i}$ on the barycenter of its three images}
update $x_i$ by the arithmetic mean $\frac{1}{|J(i)|}\sum_{k \in J(i)}\overline{x}_i$\; 
  }
Apply boundary constraints\;
Compute mesh quality Quality$(n)$.}
 \end{algorithm} 
Apart from quadrilateral meshes, the triangle mesh algorithm can be naturally extended to a volume mesh of tetrahedra by transforming the triangle faces of each tetrahedron: let $M_T$ be a mesh defined by the set $V=(x_0,\dots,x_{N-1})$ of vertices $x_i \in \mathbb{R}^3$ and the set $E=(T_0,\dots,T_{n-1})$ of tetrahedral elements $T_i=(i_0,i_1,i_2,i_3)$ with $i_j \in \left\{0,\dots,N-1\right\}$. We implemented the smoothing algorithm as given by Algorithm~\ref{alg:tetmesh}. \\
\begin{algorithm}[h]
 \caption{Tetrahedral mesh algorithm}\label{alg:tetmesh}
\KwData{nodes, elements,boundary, NumberofInnerIterations, NumberofIterations, ErrorBound}
Compute initial mesh quality $\text{Quality}(0)$\;
 \For{$i=0:\text{NumberofNodes}-1$}
{find tetrahedra $T$ which contain $x_i$. Denote the set of indices of adjacent tetrahedra by $J(i)$.}
\While{$\text{Quality}(n) -\text{Quality}(n-1) > ErrorBound$ \& $n < \text{NumberofIterations}$}{
\For{$i=0:\text{NumberofNodes}-1$}
{update vertex $x_i$ as following:
 \For{$k\in J(i)$}
{obtain $T_{k,new}$ by applying the Triangle Mesh Algorithm~\ref{alg:TriangleMesh} a NumberofInnerIteraions-times to $T_k$ as a closed mesh of four triangles,\; 
rescale $T_{k,new}=(\overline{x}_{k_0},\overline{x}_{k_1},\overline{x}_{k_2},\overline{x}_{k_3})$ to keep the volume constant.}
{update the vertex $x_i$ by the arithmetic mean $\frac{1}{|J(i)|}\sum_{k \in J(i)}\overline{x}_i$\;
  }}
Apply boundary constraints\;
Compute mesh quality Quality$(n)$}
\end{algorithm}

\subsection{Derived algorithm for hexahedral meshes}\label{s.hexmesh}

Apart from triangular and tetrahedral meshes, hexahedral meshes are maybe the most important class of meshes for applications. This is our main motivation to generalize our algorithm to hexahedral meshes. We use the fact that every hexahedron defines a octahedron whose vertices are the barycenters of the six faces of the hexahedron, see Fig.~\ref{fig:hexaeder} as an illustration. 
\begin{figure}
\centering
\includegraphics[width=0.3\textwidth]{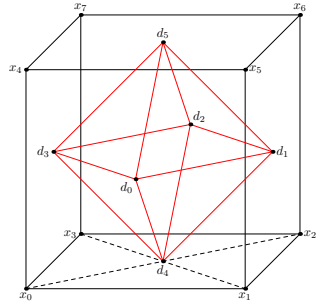}
\caption{A regular hexahedron and its dual octahedron in red.}
\label{fig:hexaeder}
\end{figure}
Conversely, every octahedron determines a hexahedron by taking the barycenters of its eight faces. In this way, we compute to every hexahedron its corresponding octahedron. This could be treated as a closed triangle mesh. Let us now define more precisely our algorithm:  \\
let $M_H$ be a hexahedral mesh defined by the set $V=(x_0,\dots,x_{N-1})$, $x_i \in \mathbb{R}^3$, of vertices, indexed counter-clockwisely, and the set $E=(H_0,\dots,H_{n-1})$ of hexahedral elements $H_i=(i_0,\dots,i_7)$ with $i_j \in \left\{0,\dots,N-1\right\}$. The smoothing algorithm is then defined by Algorithm~\ref{alg:hexmesh}. \newline
\begin{algorithm}[h]
\caption{Hexahedral mesh algorithm}\label{alg:hexmesh}
 \KwData{nodes, elements,boundary, NumberofInnerIterations, NumberofIterations, ErrorBound}
Compute initial mesh quality $\text{Quality}(0)$\;
 \For{$i=0:\text{NumberofNodes}-1$}
{find hexahedra $H$ which contain $x_i$. Denote the set of indices of adjacent hexahedra by $J(i)$.}
\While{$\text{Quality}(n) -\text{Quality}(n-1) > ErrorBound$ \& $n < \text{NumberofIterations}$}{
\For{$i=0:\text{NumberofNodes}-1$}
{update vertex $x_i$ as following:
 \For{$k\in J(i)$}
{compute to $H_k$ its dual octahedron $O_k$ whose vertices are defined by the barycenters of the faces of $H_k$\;
obtain $O_{k,new}$ by applying the Triangle Mesh Algorithm~\ref{alg:TriangleMesh} a NumberofInnerIterations-times to $O_k$ as a closed mesh of eight triangles\; 
compute $H_{k,new}=(\overline{x}_{k_0},\overline{x}_{k_1},\overline{x}_{k_2},\overline{x}_{k_3})$ by taking the barycenters of the triangles faces of $O_{k,new}$\;
rescale $H_{k,new}$ to keep its volume constant\;}
update $x_i$ by the arithmetic mean $\frac{1}{|J(i)|}\sum_{k \in J(i)}\overline{x}_i$\; 
  }
Apply boundary constraints\;
Compute mesh quality Quality$(n)$}
 \end{algorithm} 
\begin{rem} \begin{enumerate}
\item An equivalent application of our algorithm to a hexahedral mesh is the following: any hexahedron could be subdivided into four tetrahedra whose edges are given by the diagonals of the faces of the hexahedron. Apply then the algorithm to this mesh of four tetrahedra. 
\item Any platonic solid, these are tetrahedra, hexahedra, octahedra, dodecahedra and icosahedra, can be subdivided into regular tetrahedra. In this way, one could apply our algorithm to any mesh built by these polyhedra. 
\end{enumerate}    
\end{rem}
\section{Numerical results and discussion on mesh quality improvement} \label{s.numerical_results}
In the following we shortly discuss the numerical results of our smoothing algorithm which we implemented in C++. The implementation is straight-forward as described above and not optimized with regard to run time and storage. For this reason, the run times for our algorithm depicted below should be taken as an upper bound for the run time with potential to diminish considerably. The geometric triangle transformation is firstly applied to all elements and the nodes of these iterated elements are saved as intermediate nodes. Then the nodes are updated by the barycenter of these intermediate nodes. We have chosen two simple triangle, one quadrilateral, one tetrahedral and one hexahedral meshes as examples and display the mean quality improvement. This article does not focus on computational details, so this section serves more as a proof of concept that the geometric element transformation, described in this article, could be the base of a very efficient mesh smoothing algorithm which has a comparable run time to Laplace-type algorithms, but obtains usually better quality results. \\
More precisely, as Laplace algorithm we choose a so-called \emph{SmartLaplace} which is namely the classical approach with the addition that the inversion of elements is inhibited. It is the most appropriate and fair choice as we do the same to prevent inverted elements within the geometric mesh smoothing algorithm. \\
Apart from the Laplace algorithm we use for the volume meshes a global optimization method implemented in MESQUITE 2.3.0 (Mesh Quality Improvement Toolkit) to compare our obtained smoothing results with regard to element and mean quality improvement and run time. As objective function for the global optimization we select the inverse of the mean ratio quality (geometric mean of the element mean ratio quality) and as numerical optimization scheme the feasible Newton method (see user's guide \cite{M14} for details). The source code was in all cases compiled using g++ under Linux. For all testing we use the same personal computer equipped with a quad-core-processor (Intel(R) Core(TM) i7 CPU 870 @293 GHz, 1197 MHz). The algorithms are in all cases stopped if the mean mesh quality improvement was less than $10^{-4}$ during the last iteration step, this means, that the resulting meshes are all converged. All initial sample meshes except from the quadrilateral mesh are valid meshes, that is, without inverted elements. For the quadrilateral mesh we have observed an untangling due to our algorithm so that we have chosen an invalid mesh.      

\subsection{Quality measures}
We briefly introduce the quality measures which we use for quality assessment of the mesh smoothing algorithms. We have chosen the mean ratio quality as standard and widely used measure for the shape of an element and the arithmetic mean of the element quality measure to measure the quality of the entire mesh.  
\paragraph{Triangle and quadrilateral mesh}
As quality measure $q_{\Delta}$ for a single triangle element we use the mean ratio quality measure given by 
\begin{align*}
q_{\Delta}&=\frac{2\det(S)}{\trace(S^tS)}, \quad\mbox{where}\;S=D(\Delta)W^{-1},\\
D(\Delta) &= (x_1 - x_0, x_2-x_0) \quad \Delta = (x_0,x_1,x_2) \in (\mathbb{R}^2)^3,\\
W &= \begin{pmatrix}\sfrac{1}{2} &-1\\ \sfrac{\sqrt{3}}{2}& 0 \end{pmatrix}.
\end{align*}
For a single quadrilateral element $q_{\Delta}$ we use the edge length ratio. 
The \textit{quality measure for a triangle or quadrilateral mesh} $V=(\Delta_0,\dots,\Delta_{|V|-1})$ is then the arithmetic mean of the quality measure $q_{\Delta}$ for every element $\Delta \in V$: $q_V=\frac{1}{\left|V\right|}\sum_{\Delta \in V}q_{\Delta}.$
\paragraph{Tetrahedral mesh}
Let $T=(x_0,x_1,x_2,x_3)$ with $x_i \in \mathbb{R}^3$ be a tetrahedron. As quality measure $q_T$ for $T$ we use the \emph{mean ratio quality measure} which is defined as following (see \cite{Knupp2001}):
\begin{align*}
q_T(T)&=\frac{3 \det(S)^{2/3}}{\trace(S^tS)}, \quad S=D(T)W, \;\mbox{where}\\
D(T)&=(x_1-x_0, x_2-x_0, x_3-x_0),\quad W=\begin{pmatrix}1& 1/2 &1/2\\0& \sqrt{3}/2&\sqrt{3}/6\\0&0&\sqrt{2/3}\end{pmatrix}. 
\end{align*}
As quality measure for a tetrahedral mesh $V=(T_0, \dots, T_{|V|-1})$ we use the mean quality measure of every element: 
$q_V=\frac{1}{|V|}\sum_{T \in V}q_T(T).$ 
\paragraph{Hexahedral mesh}
Let $H=(x_0,\dots,x_7)$ with $x_i \in \mathbb{R}^3$ be a hexahedron. As quality measure $q_H$ for $H$ we use the \emph{mean ratio quality measure} which is defined using a subdivision of $H$ into eight tetrahedra $T_0,\dots,T_7$ given by $T_0=(x_0,x_3,x_4,x_1), T_1=(x_1,x_0,x_5,x_2), T_2=(x_2,x_1,x_6,x_3),T_3=(x_3,x_2,x_7,x_0), T_4=(x_4,x_7,x_5,x_0), T_5=(x_5,x_4,x_6,x_1), T_6=(x_6,x_5,x_7,x_2)$ and $T_7=(x_7,x_6,x_4,x_3)$. The quality measure of $H$ is then defined as the average of the eight values of the quality measure for the internal tetrahedra, with the difference that $W$ is set to identity, so we get:  
\begin{align*}
q_H(H)&=\frac{1}{8}\sum_{k=1}^8\frac{3 \det(S_k)^{2/3}}{\trace(S_k^tS_k)}, \quad S_k:=D(T_k), \;\mbox{where}\\
D(T_k)&=(x_{k,1}-p_{k,0}, x_{k,2}-x_{k,0},x_{k,3}-x_{k,0}), \quad \mbox{with}\; T_k=(x_{k,0},x_{k,1},x_{k,2},x_{k,3}).
\end{align*}

\subsection{Two-dimensional meshes}
\subsubsection{Example 1: triangulated unit square}
The very first example is a randomly generated triangulation of the unit square, see Fig.~\ref{fig:square}. In Fig.~\ref{fig:adaptive} we have studied the influence of the choice of adaptive parameters $(\alpha_0,\alpha_1)$, and following those results, we choose $(\alpha_0,\alpha_1)=(0.1,0.15)$ and call the so adapted triangle mesh algorithm~\ref{alg:TriangleMesh} \emph{GETMe adaptive}. 
Due to the topology of the present example, the standard geometric algorithm would produce invalid elements if not actively impeded. Consequently, as several elements have to be reset in each iteration step to prevent their flipping this decelerates the algorithm. Because it does not at all produce invalid elements, the usage of adaptive parameters is for the present mesh quite effective and produce a slightly better smoothing result faster than the standard algorithm (see Fig.~\ref{fig:quality1} or Fig.~\ref{fig:vergleich_tri}). 
\begin{figure}[htbp]
\centering

\begin{minipage}{0.47\textwidth}
\flushleft
 \includegraphics[width=\textwidth]{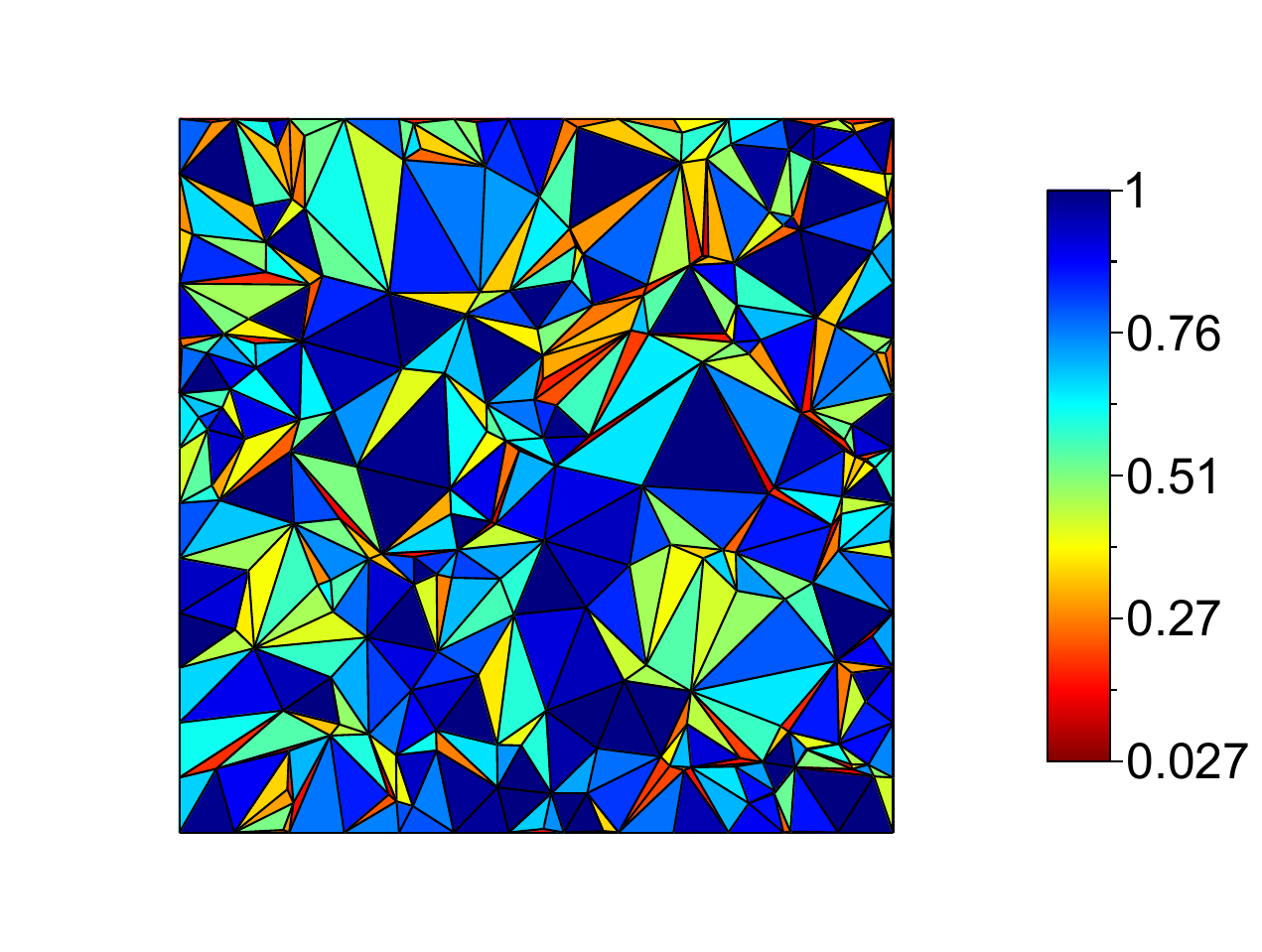}
\end{minipage}
\begin{minipage}{0.47\textwidth}
\flushright
\includegraphics[width=\textwidth]{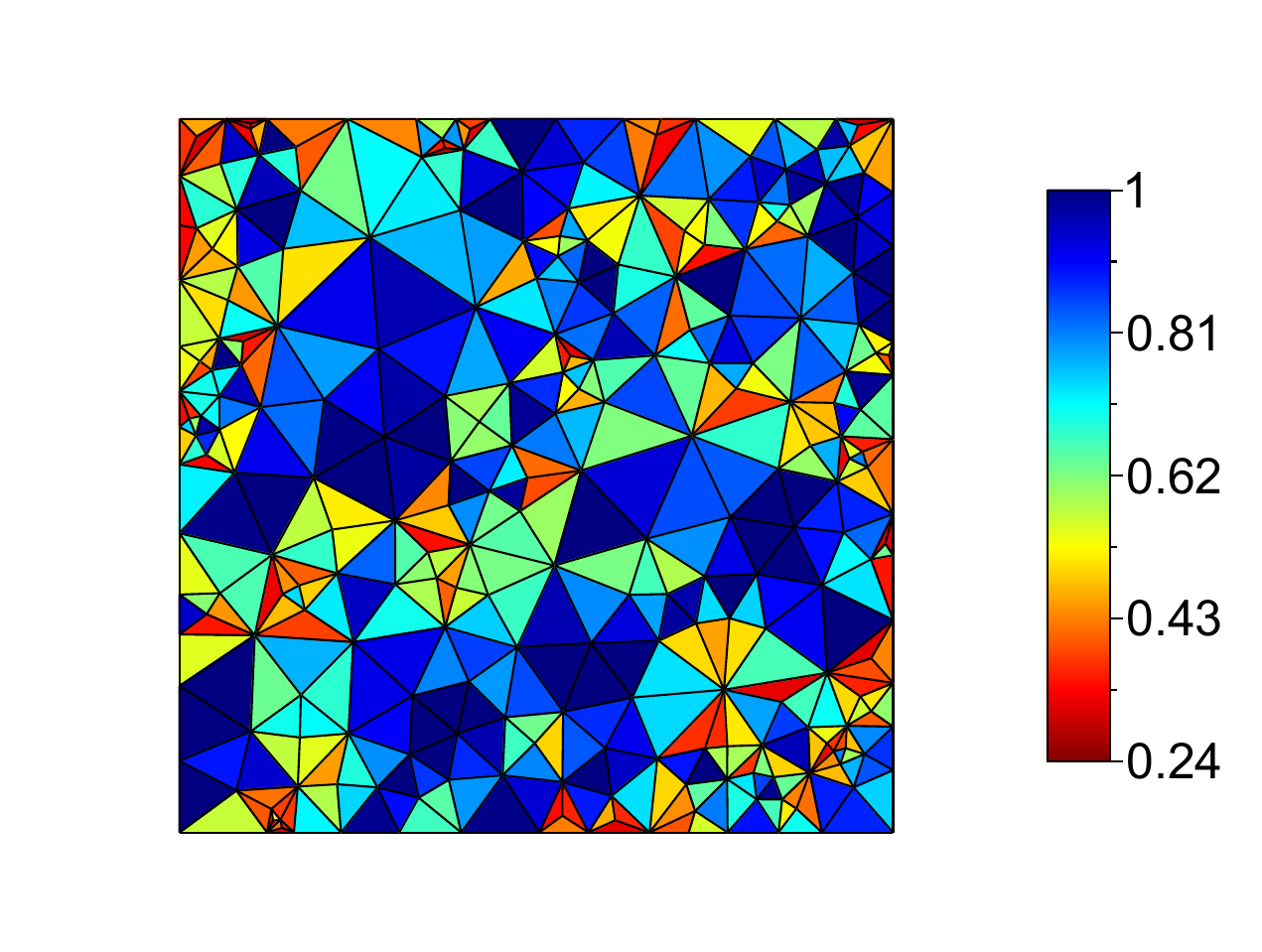}
\end{minipage}
\caption{On the left, initial triangulation with mean ratio quality $q_V=0.609$ (arithmetic mean) and $q_V=0.528$ (geometric mean), on the right, the triangulation smoothed by \emph{GETMe adaptive} with $q_V=0.826$ (arithmetic mean) and $q_V=0.806$ (geometric mean).}
\label{fig:square}
\end{figure} 

\begin{figure}[htbp]
\centering
\begin{minipage}{0.48\textwidth}
\begin{flushleft}
\includegraphics[width=8cm,height=5.5cm,keepaspectratio]{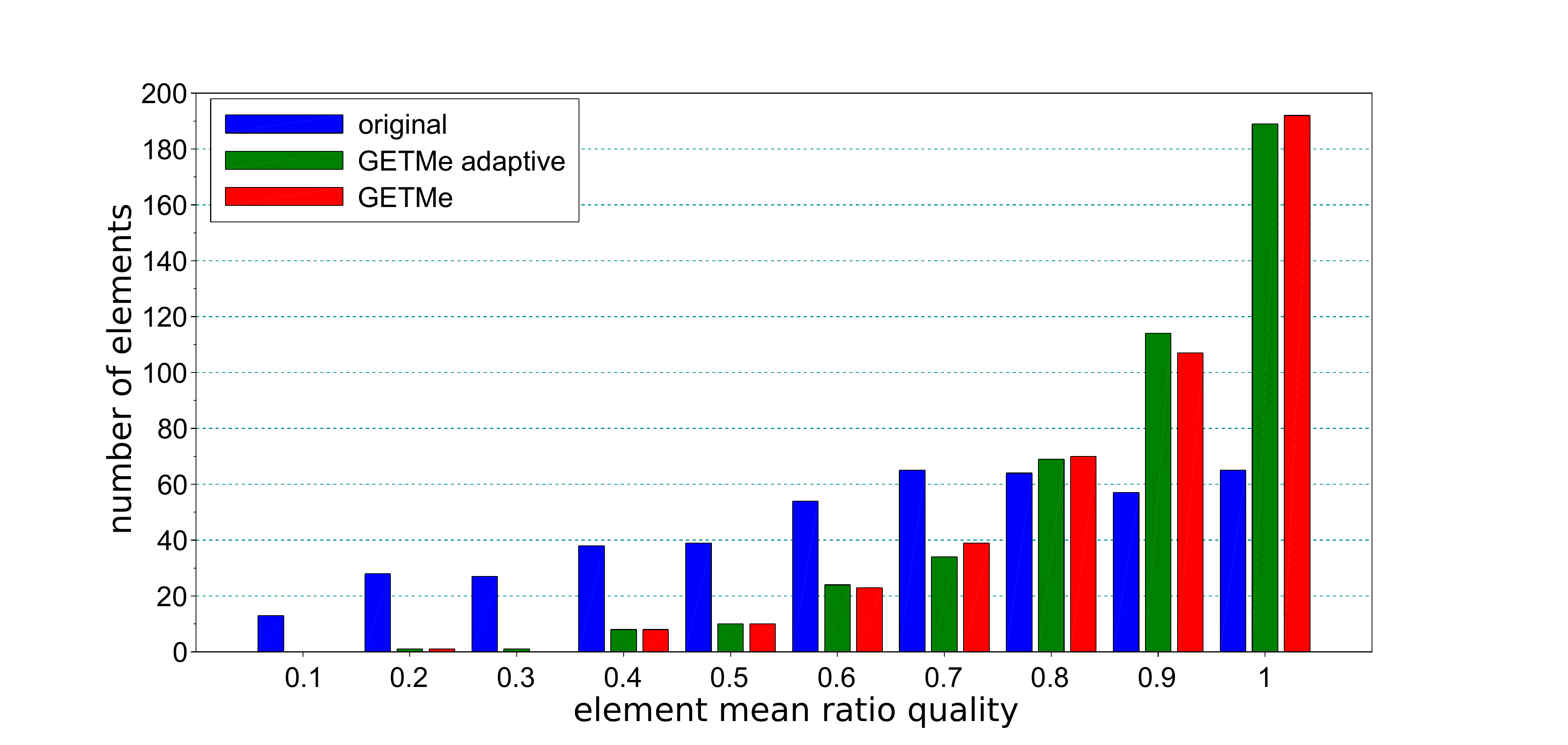}
\end{flushleft}
\end{minipage}
\begin{minipage}{0.48\textwidth}
\flushright
\includegraphics[width=7.5cm,height=5cm,keepaspectratio]{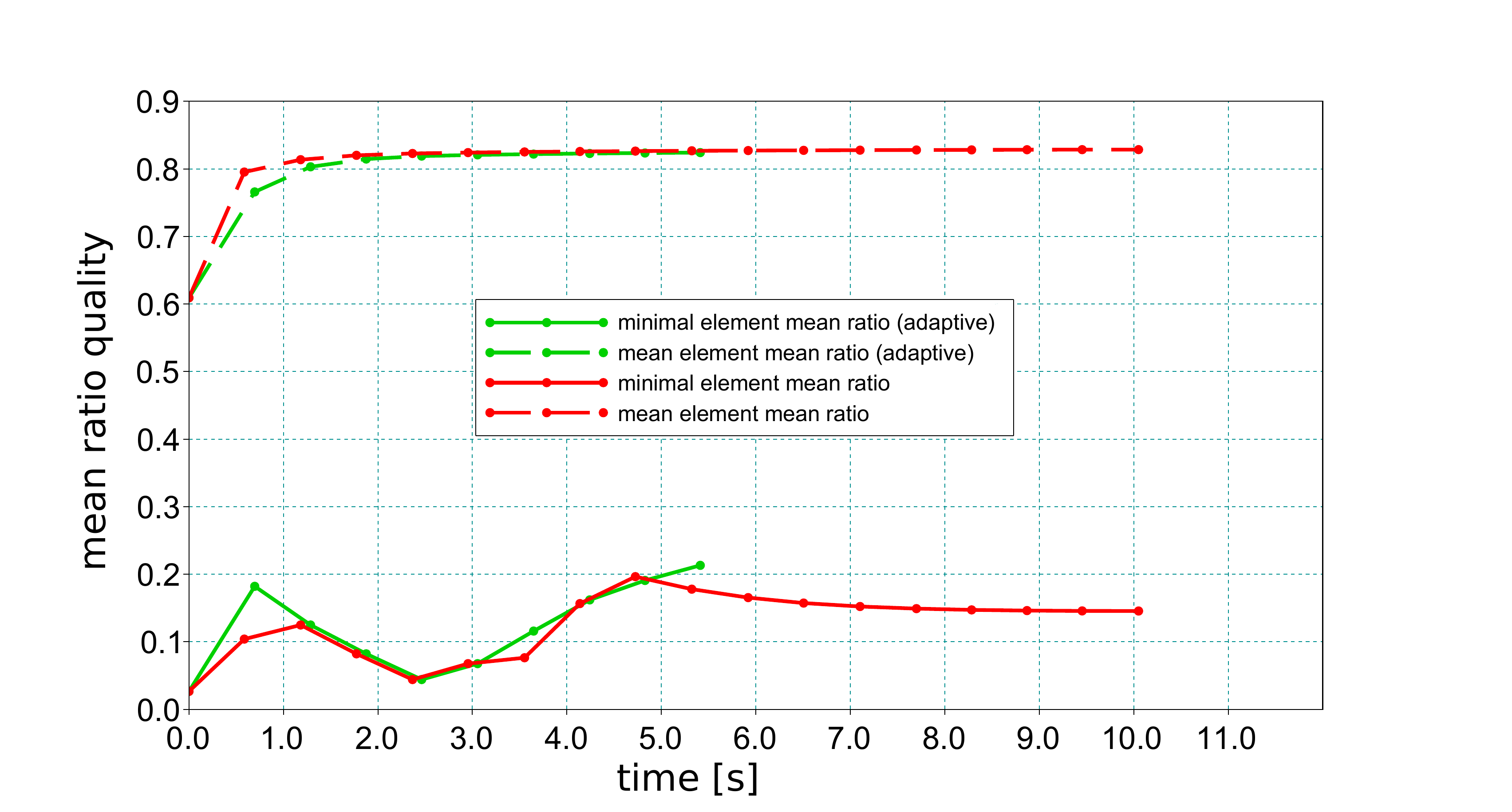}
\end{minipage}
\caption{Mean quality element measure $q_{V}$ before and after the smoothing of the mesh in Fig.~\ref{fig:square} and improvement of the minimal element and the mean mesh quality over run time (the marks corresponds to iteration steps), both for the standard triangle mesh Algorithm~\ref{alg:TriangleMesh} and the adapted one with adapted parameters $(\alpha_0,\alpha_1)=(0.1,0.15)$ (see Sec.~\ref{sec:adaptive})}.
\label{fig:quality1}
\end{figure} 

\subsubsection{Example 2: planar disk}
For this example, we shortly study the performance of the standard and adaptive triangle mesh Algorithm~\ref{alg:TriangleMesh}, called \emph{GETMe} and \emph{GETMe adaptive}, respectively. The triangle sample mesh displayed in Fig.~\ref{fig:trimesh} is a planar triangle mesh which was generated by triangulating a planar disk. Due to its topology, all algorithms converge to the same optimal triangulation (see Fig.~\ref{fig:trimesh} on the right). Looking at Fig.~\ref{fig:vergleich_tri}, one infers that the use of adapted parameters is a viable possibility to accelerate the convergence.\\
For the moment, we have chosen globally constant parameters for the whole mesh, but like an adaptive step size strategy one should choose the parameters according to the distortion of the mesh regions, for example checking the distortions of the one- or two-ring neighborhood for each triangle. It seems in any case recommended to start with small parameters and to accelerate the convergence rate, that is, to increase the adaptive parameters, with growing mesh quality. Further, one could adapt the parameters during each iteration.   
\begin{figure}[htbp]
\centering
\includegraphics[width=\textwidth]{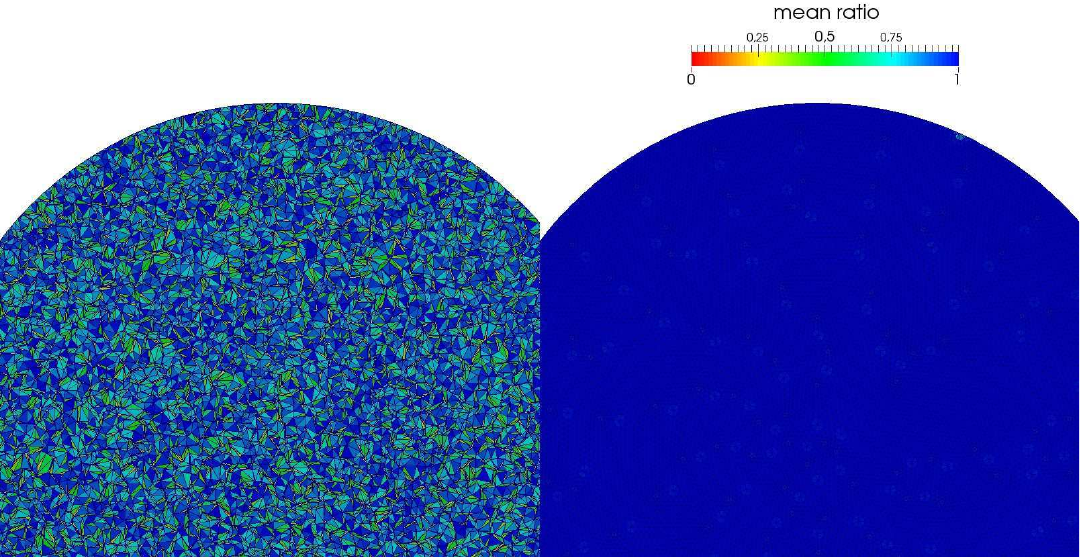}
\caption{Planar triangle mesh of 38560 elements with mean edge ratio $q_V=0.399$ (arithmetic mean) and $q_V=0.185$ (geometric mean) and minimal element quality $\min q_{\Delta}=10^{-6}$ smoothed using adapted parameter $(\alpha_0,\alpha_1)=(0.1,0.15)$ (see Sec.\ref{sec:adaptive}) to $q_V=0.941$ (arithmetic mean and geometric mean) with minimal element quality $\min q_{\Delta}=0.002$.}
\label{fig:trimesh}
\end{figure} 
%
\begin{figure}[htbp]
\centering
\begin{minipage}{0.48\textwidth}
\begin{flushleft}
\includegraphics[width=7.5cm,height=5cm,keepaspectratio]{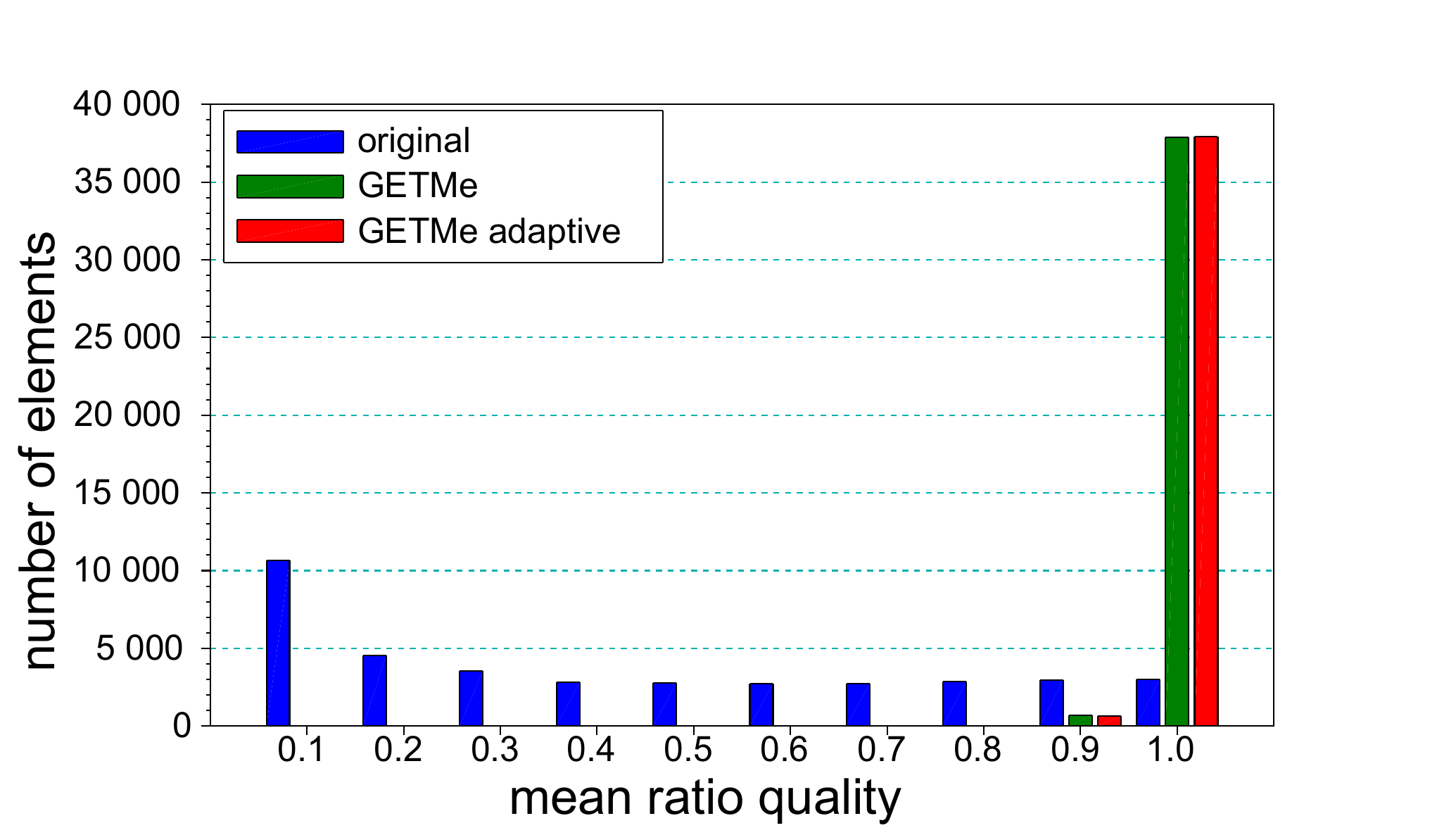}
\end{flushleft}
\end{minipage}
\begin{minipage}{0.48\textwidth}
\flushright
\includegraphics[width=7.5cm,height=4.3cm,keepaspectratio]{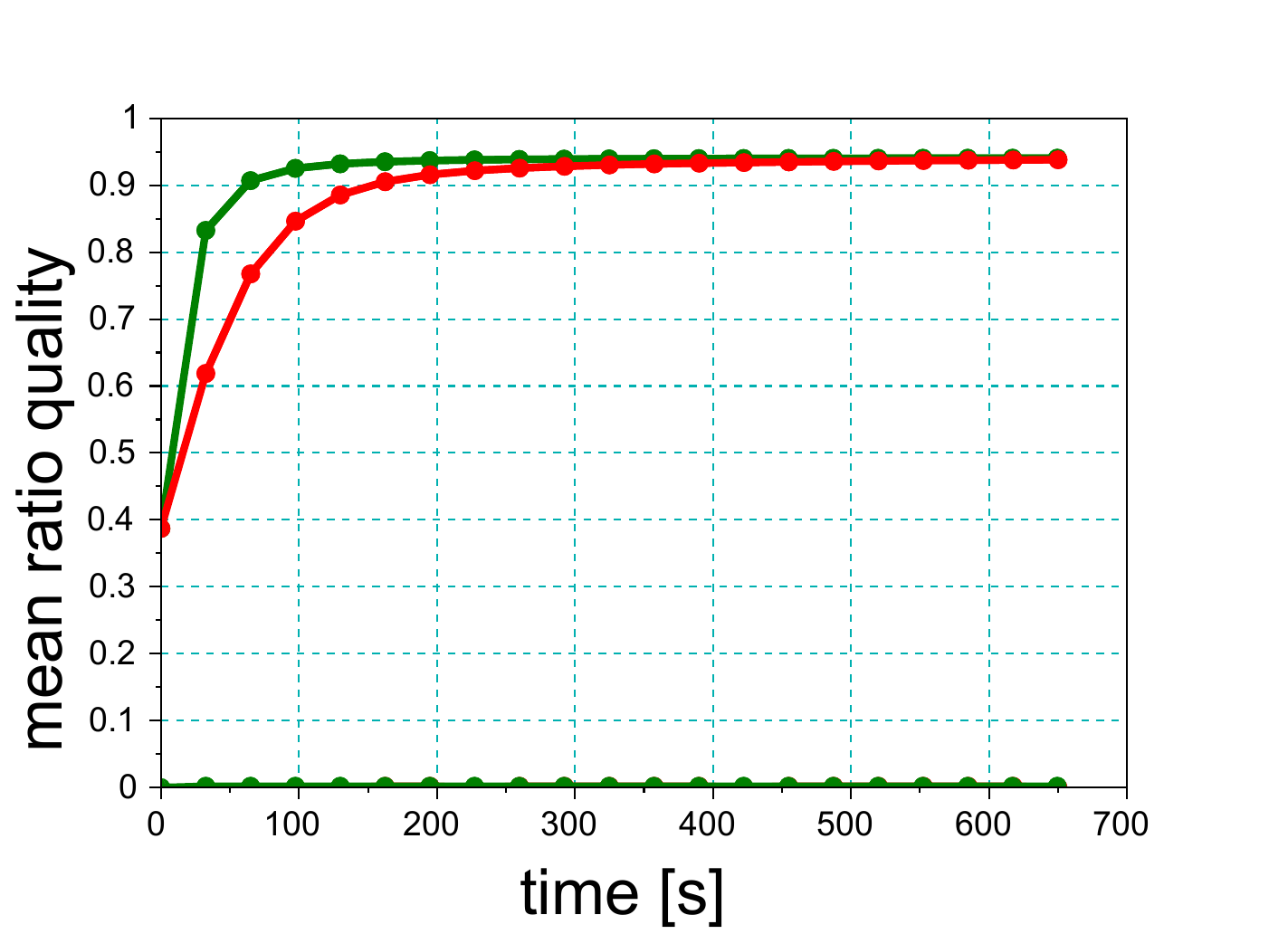}
\end{minipage}
\caption{Histogram of mesh quality improvement and comparison of mesh quality improvement for the triangulated disk of Fig.~\ref{fig:trimesh} using standard and adaptive parameters. }
\label{fig:vergleich_tri}
\end{figure}
\subsubsection{Example 3: planar quad mesh}
As a simple quadrilateral example mesh (see on the left in Fig.~\ref{fig:quadmesh}) we choose a square with an omitted inner circle which is randomly decomposed into quadrilaterals and apply Algorithm~\ref{alg:quadmesh}. This sample mesh is taken from Mesquite called \emph{hole in square}. We decide heuristically to apply the triangle mesh algorithm~\ref{alg:TriangleMesh} ten times to each subtriangle of each quadrilateral (see definition of the algorithm) and determine as best adaptive parameters again $(\alpha_0,\alpha_1)=(0.1,0.15)$. The results are shown in Fig.~\ref{fig:quadmesh}. 
\begin{figure}[htbp]
\centering

\includegraphics[width=\textwidth]{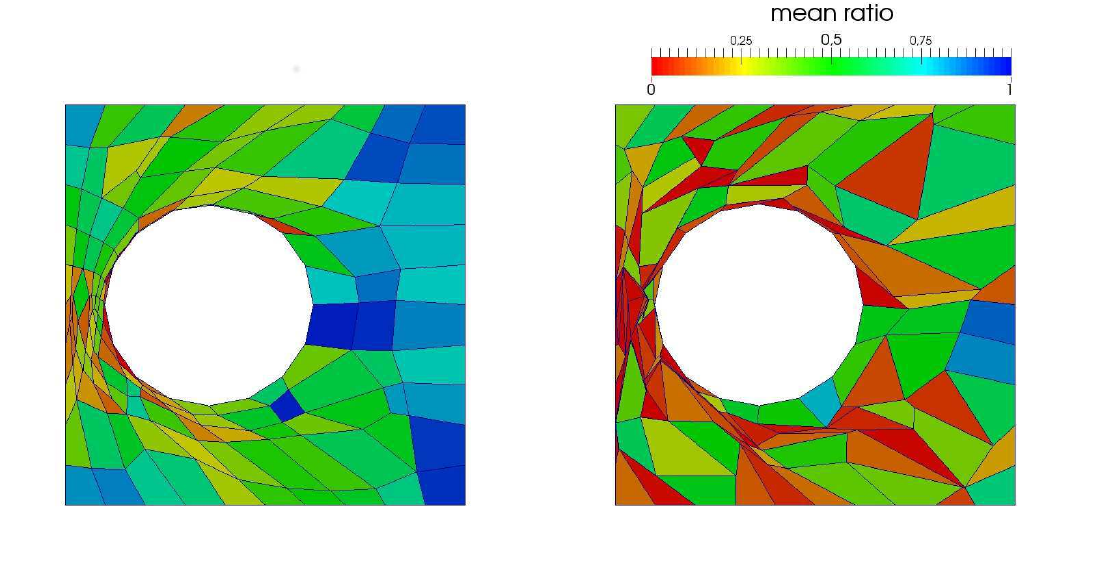}
\caption{On the left, the smoothed mesh with $q_V=0.447$ and on the right, the initial quadrilateral mesh with a mesh quality $q_V=0.259$. The color mapping is taken from ParaView (see \cite{VTK}).}
\label{fig:quadmeshing}
\end{figure}
\begin{figure}
\includegraphics[width=0.5\textwidth]{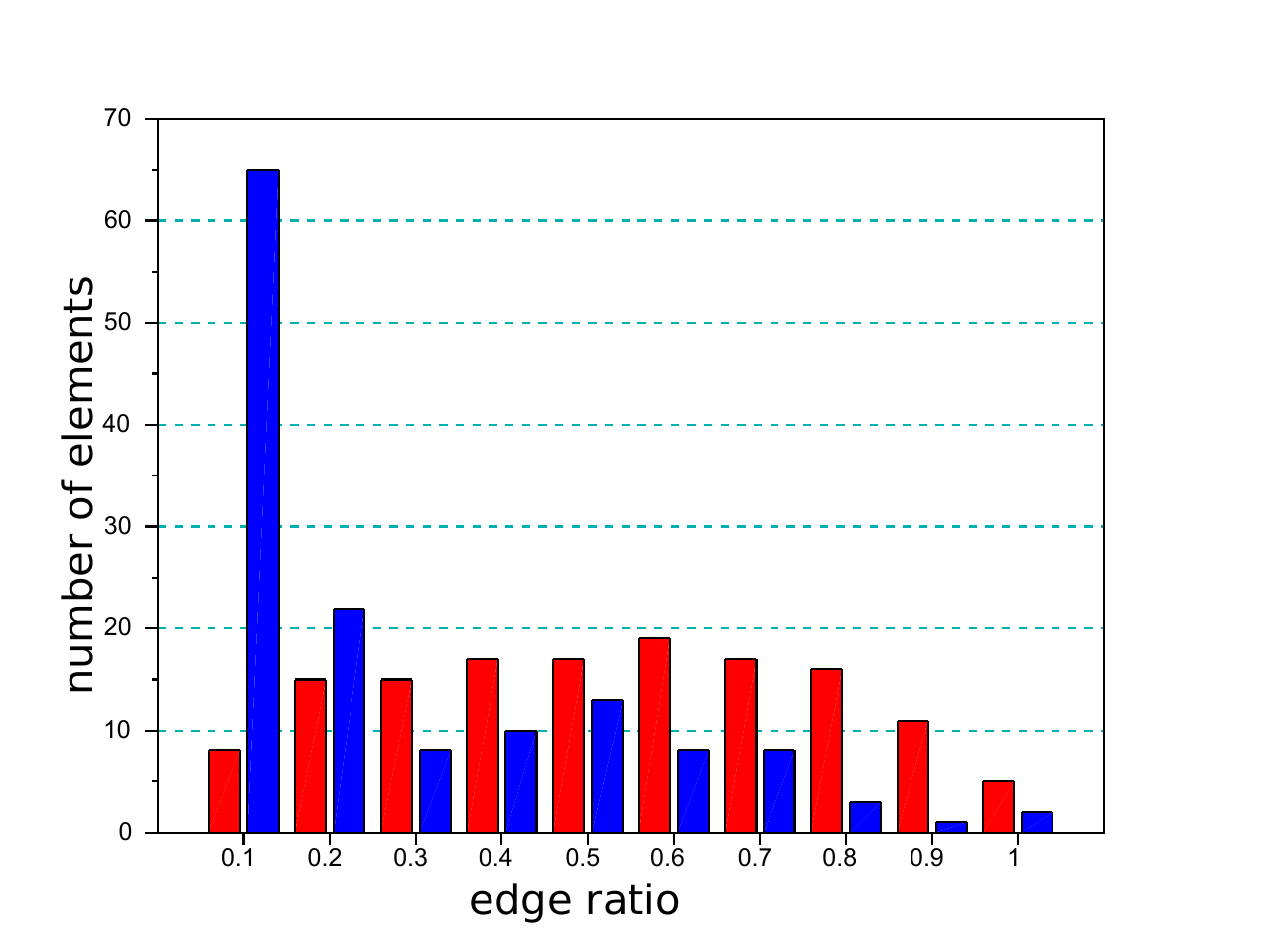}
\caption{Histogram of the original (in blue) and the smoothed mesh (in red) in Fig.~\ref{fig:quadmeshing}.}
\label{fig:quadmesh}
\end{figure} 
\subsection{Three-dimensional meshes}
\subsubsection{Example 1: tetrahedral mesh}
Our tetrahedral sample mesh (Fig.~\ref{fig:prt4}) is based on a model provided by the 3D meshes research database GAMMA maintained by INRIA which was meshed by 82958 tetrahedral elements.
We apply the tetrahedral mesh Algorithm~\ref{alg:tetmesh} with adaptive parameters $(\alpha_0,\alpha_1)=(0.6,0.6)$ and as comparison a Laplace algorithm and Mesquite as described above to the same initial mesh. 
\begin{figure}[htbp]
\centering
\includegraphics[width=\textwidth]{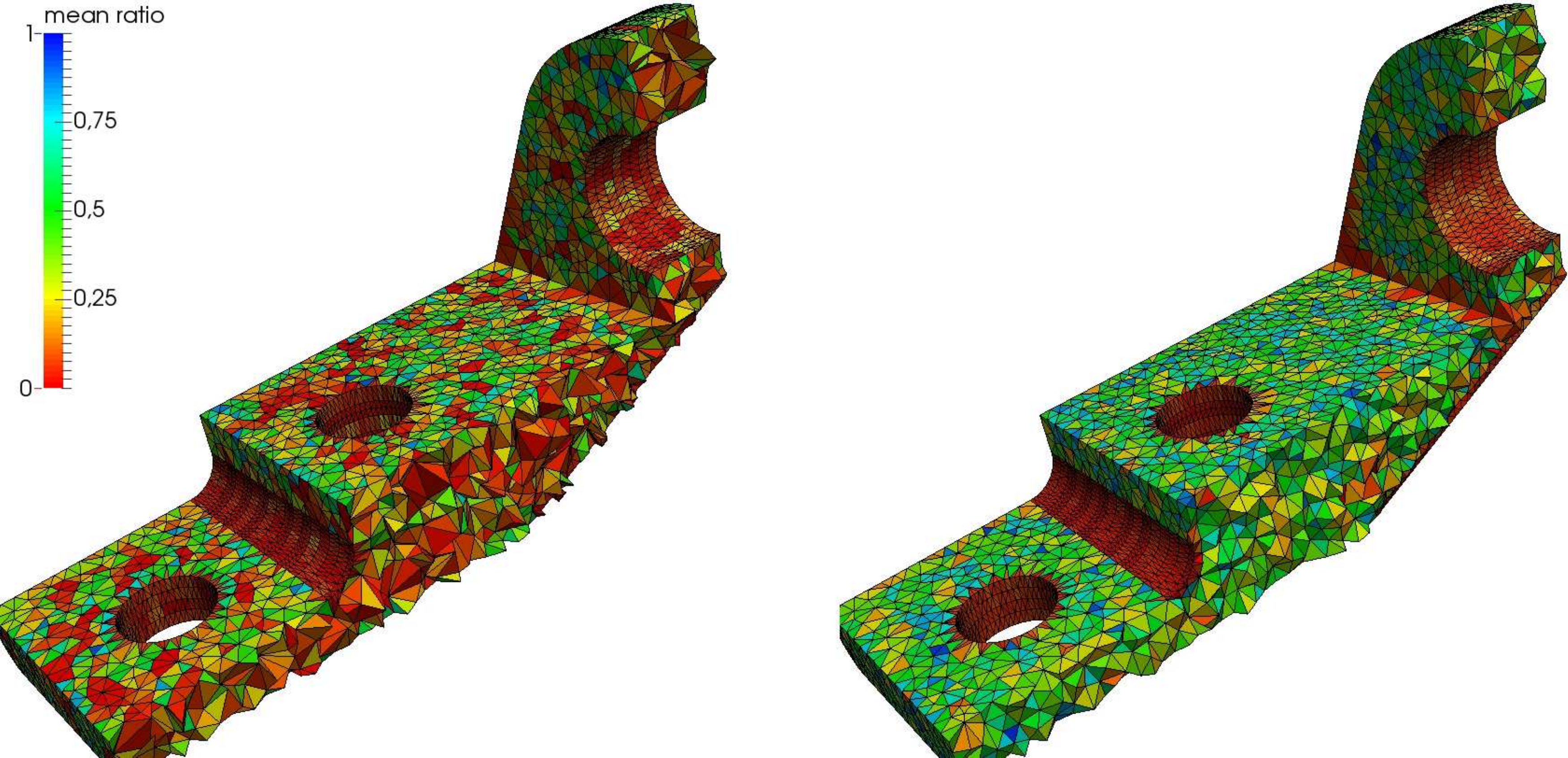}
\caption{Cross section of a tetrahedral mesh with 82985 elements with initial mean quality $q_V=0.484$ (geometric mean) and minimal element quality $\min q_T =4.9*10^{-3}$ improving to $q_V=0.741$ (geometric mean) and minimal element quality $\min q_T = 0.002$ with adaptive parameters $(\alpha_0,\alpha_1)=(0.6,0.6)$. The color mapping is taken from ParaView (see \cite{VTK}).}
\label{fig:prt4}
\end{figure}
%
Our smoothing result shown in Fig.~\ref{fig:vergleich} is nearly the same than the one obtained with Mesquite, but the run time is slightly longer.
\begin{figure}[htbp]
\centering
\begin{minipage}{0.47\textwidth}
\includegraphics[width=1.2\textwidth]{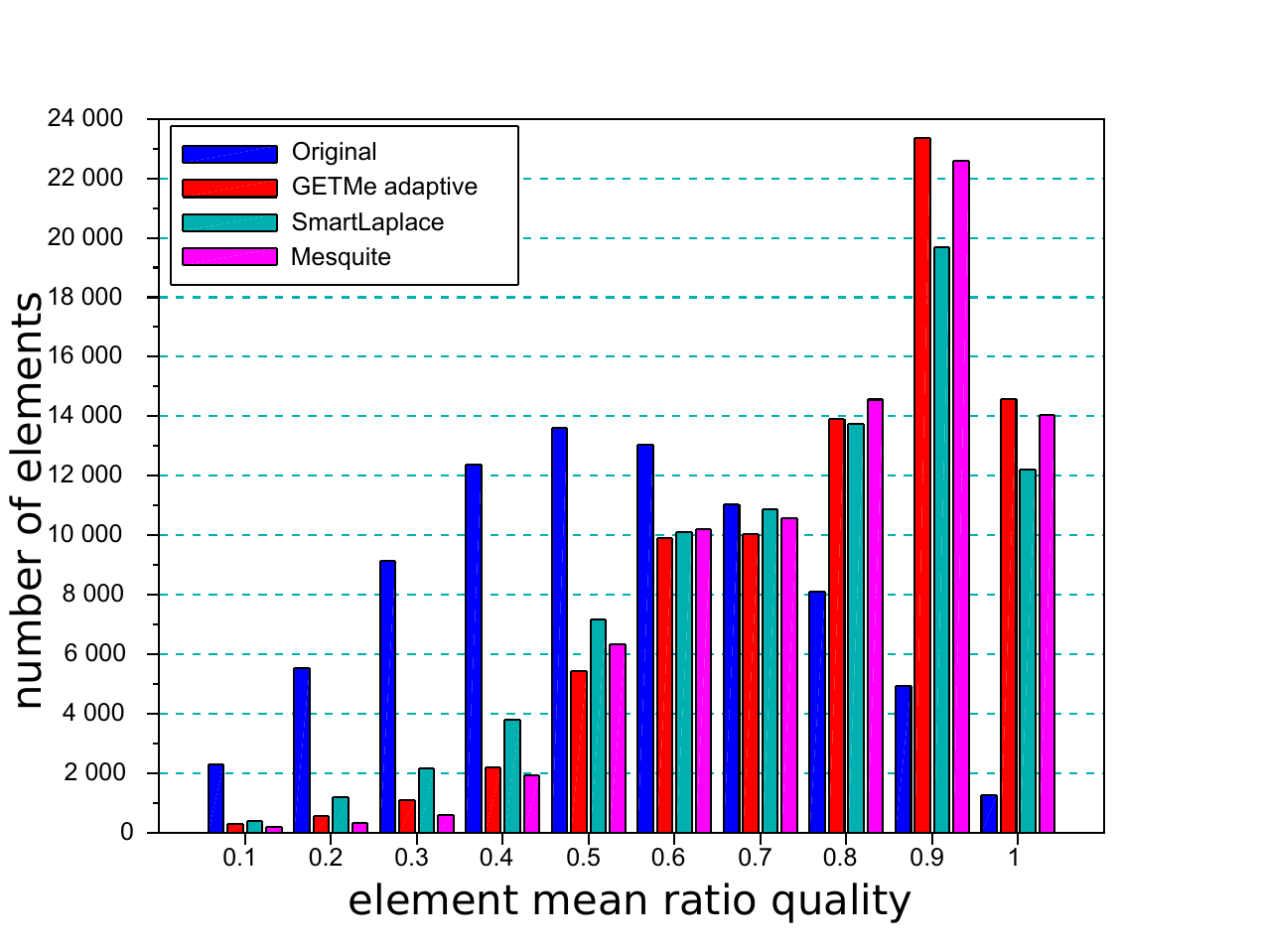}
\end{minipage}
\begin{minipage}{0.47\textwidth}
\includegraphics[width=1.2\textwidth]{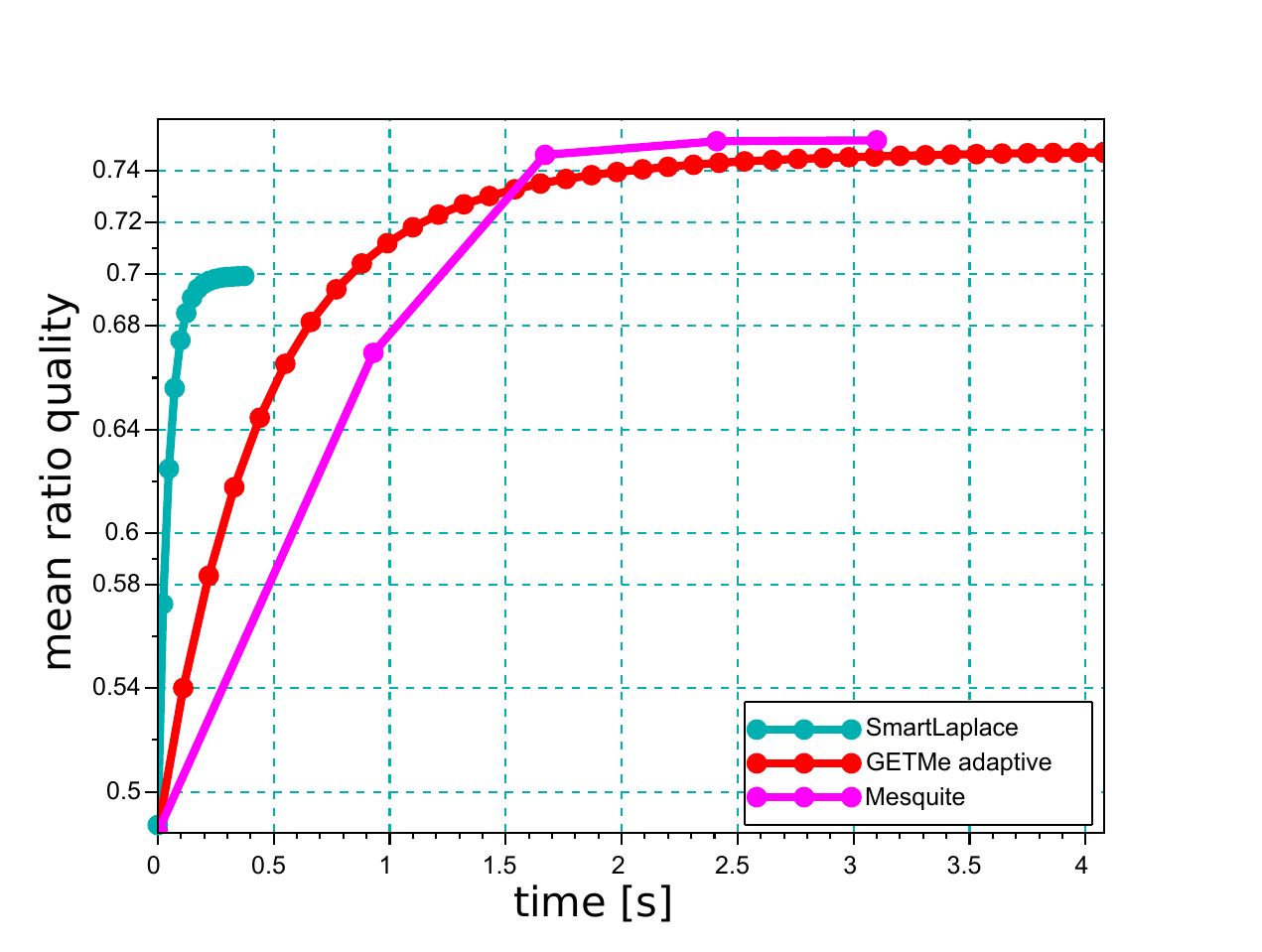}
\end{minipage}
\caption{Histogram of mesh quality improvement for the tetrahedral mesh of Fig.~\ref{fig:prt4}. The minimal element quality obtained with Laplace is $\min q_T = 9.5*10^{-5}$ and with Mesquite $\min q_T=0.06$.}
\label{fig:vergleich}
\end{figure}

\subsubsection{Example 2: hexahedral mesh}
As hexahedral sample mesh (see Fig.~\ref{fig:unittest_hexahedra}) we choose a wheel bearing model which was generated by a sweep approach applied to small sub-parts of the mesh (see \cite{KnuppSweep1998}) and then meshed by 67055 hexahedral elements. The hexahedral mesh smoothing algorithm~\ref{alg:hexmesh} was applied to its dual octahedral mesh as described in \ref{s.hexmesh}. The quality assessment displayed in Fig.~\ref{fig:histogram_hexaeder} shows that our algorithm delivers a result between the ones obtained by Laplace and Mesquite.
\begin{figure}[htbp]
\centering
\includegraphics[width=\textwidth]{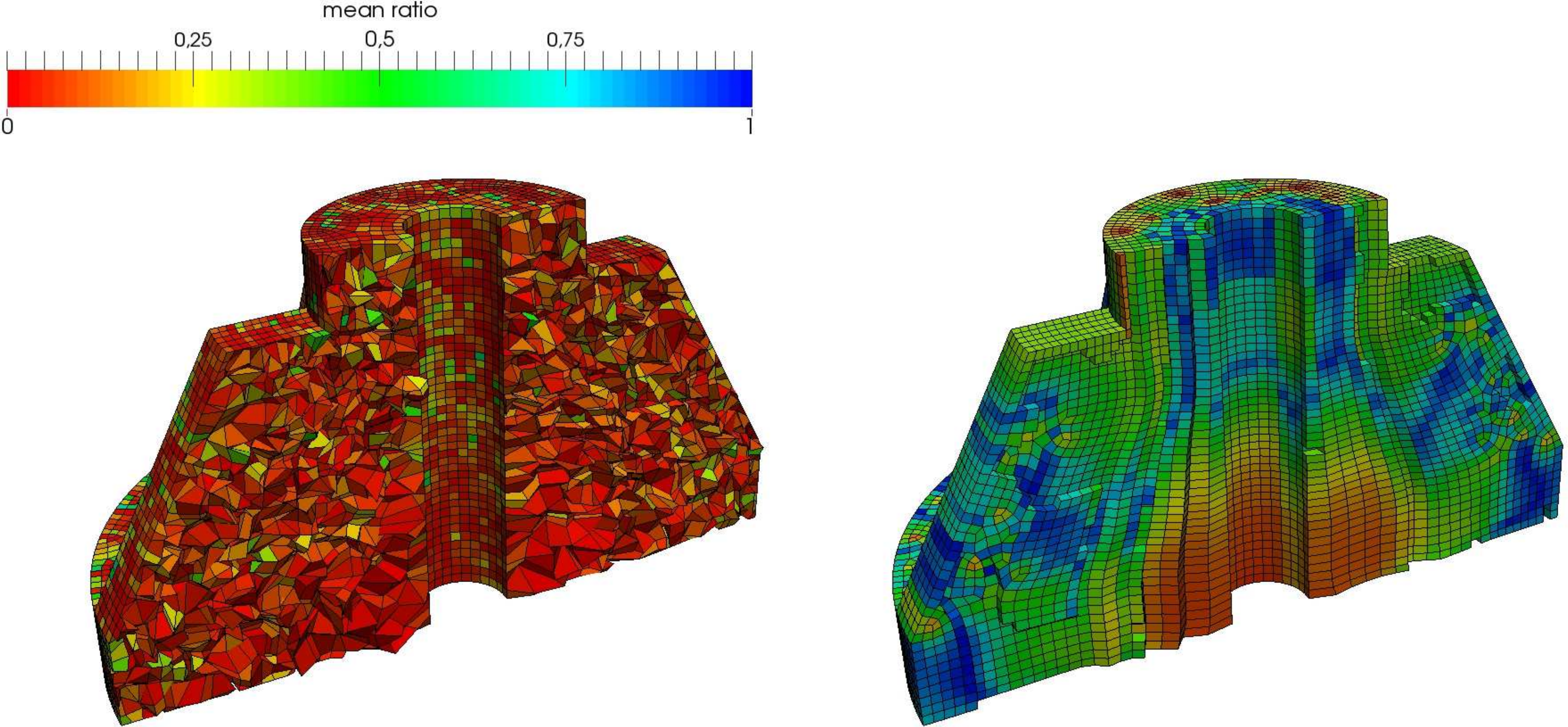}
\caption{Cross section of wheel bearing model of 67055 hexahedral elements with initial mean quality $q_V=0.46$ (geometric mean) and minimal element quality $\min q_H = 0.061$ improving to $q_V=0.92$ (geometric mean) and minimal element quality $\min q_H = 0.244$ with adaptive parameters $(\alpha_0,\alpha_1)=(0.7,0.7)$. For comparison, without the use of adaptive parameter we have reached a mean quality of $q_V=0.90$ (geometric mean) with minimal element quality $\min q_H = 0.102$.}
\label{fig:unittest_hexahedra}
\end{figure}
%
%
 
\begin{figure}[htbp]
\begin{minipage}{0.47\textwidth}
\includegraphics[width=1.2\textwidth]{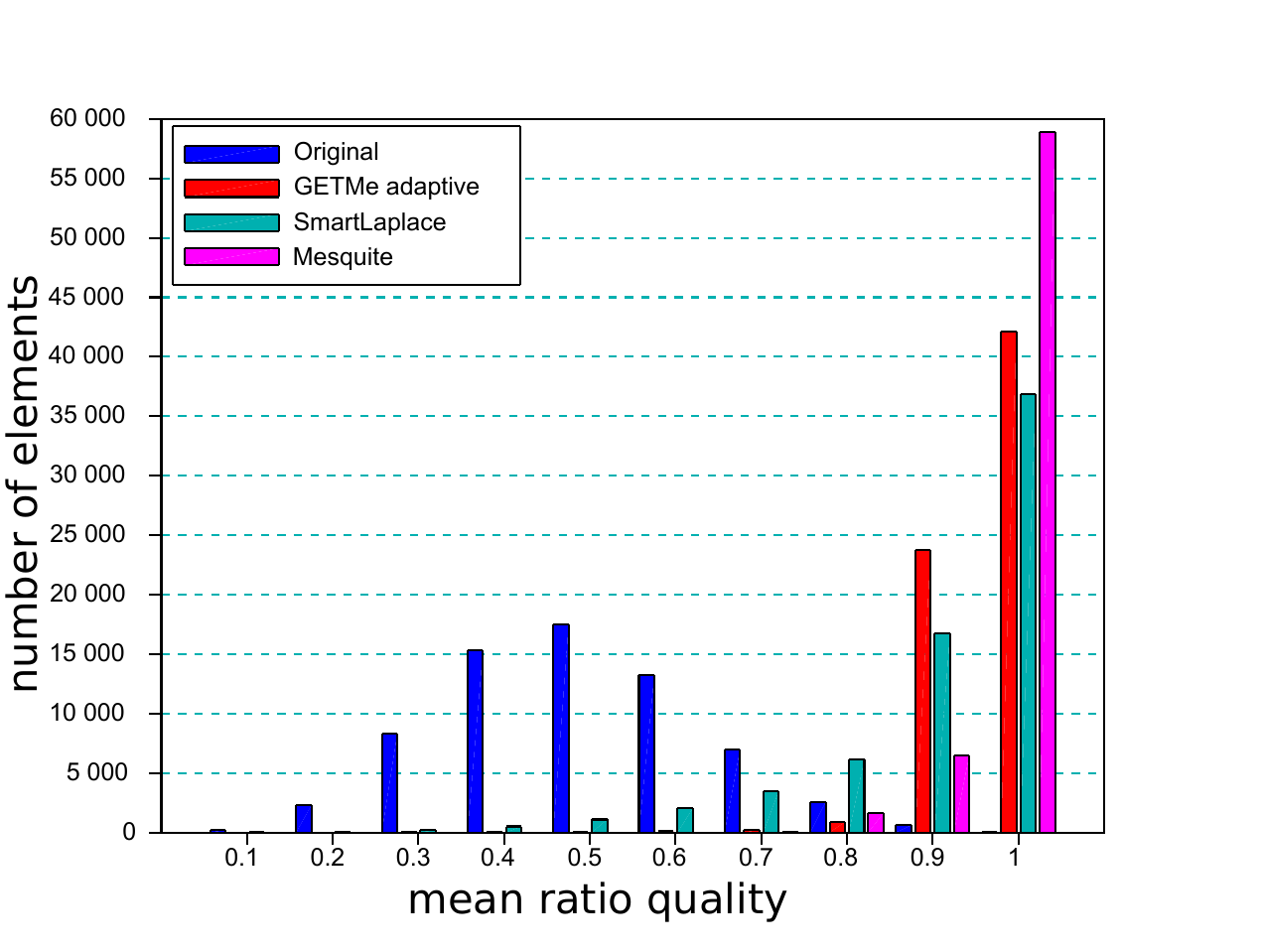}
\end{minipage}
\begin{minipage}{0.47\textwidth}
\includegraphics[width=1.2\textwidth]{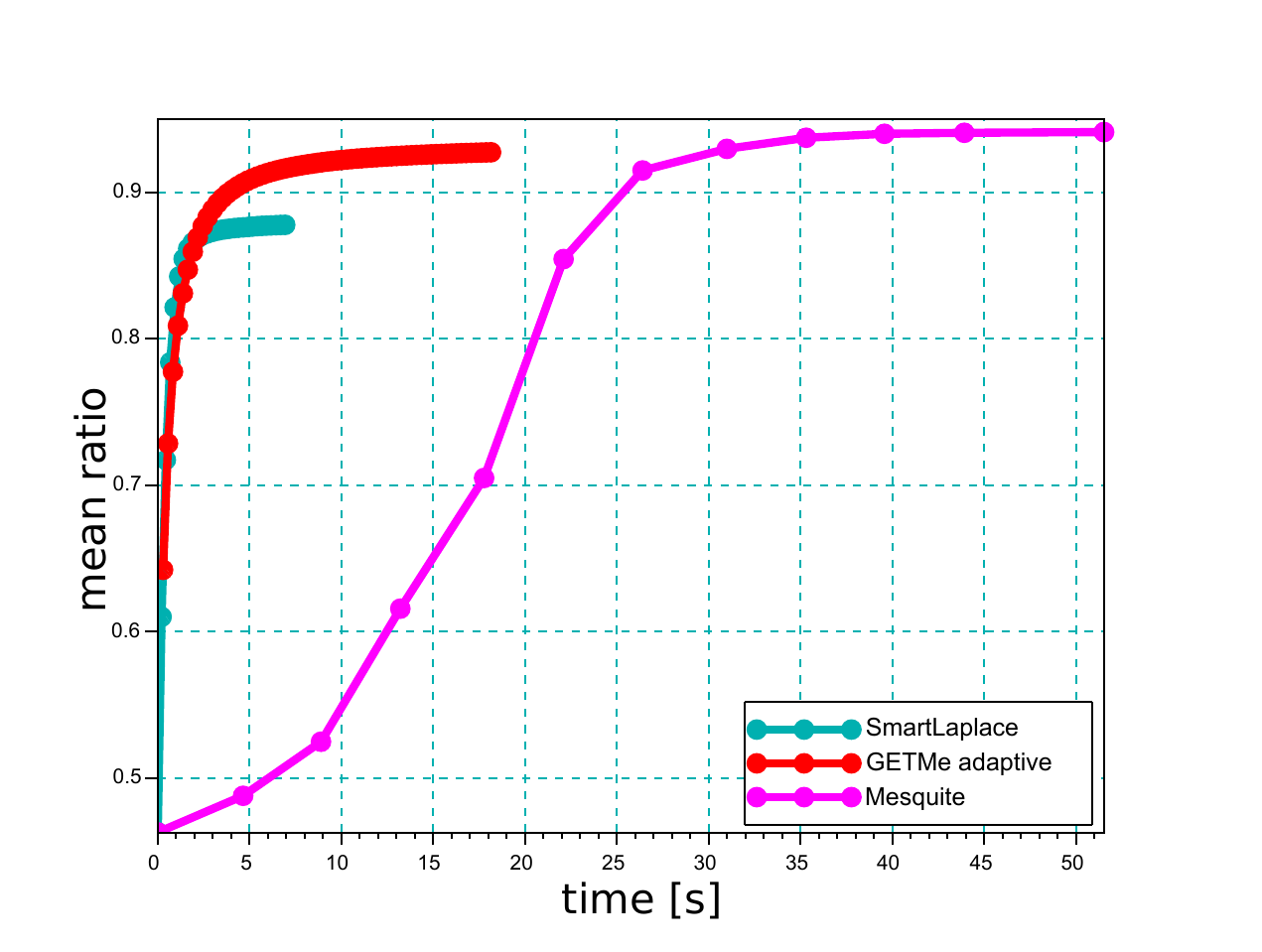}
\end{minipage}
\caption{Comparative histogram of the smoothing results by Mesquite and by SmartLaplace for the hexahedral mesh of Fig.~\ref{fig:unittest_hexahedra} on the left and comparison of mesh quality improvement of the global mesh. The minimal element quality for Laplace is $\min q_H = 0.069$ and for Mesquite $\min q_H = 0.612$. }
\label{fig:histogram_hexaeder}
\end{figure}
\begin{rem}
The figures of the three-dimensional meshes are created by ParaView 4.0.1 using the mesh quality filter. The color mapping is the quality measure called \emph{Shape} which is based on the Shape measure by Knupp for each respective volume element pre-implemented in ParaView (see \cite{VTK} where the exact definitions of the available mesh quality measures are given). We have always used the same with the maximal limits from 0 to 1.
\end{rem}  
%

%
\FloatBarrier
\section{Concluding Remarks}
The presented geometric triangle transformation exhibits interesting mathematical properties and is, at a first glance, appropriate for a broad usage as base of a mesh smoothing algorithm. We did not explore further the possibilities to implement the algorithm in the best and most efficient way possible as this article focuses on mathematical considerations such that the run time could be clearly reduced. Additionally, as an element-wise transformation it offers the possibility to be quite easily parallelizable. Further, one could implement a dynamical choice of the adaptive parameters depending on the distortion of every single element or of mesh regions to improve the performance of the smoothing algorithm. Although the tests with tetrahedral meshes have not yet been convincing, the much better results for planar triangle meshes and the hexahedral mesh make hope that one can find an intelligent way to control the rate of the element-wise convergence of the algorithm in every step such that the overall result improves considerably.\\ Another important aspect of future practical application will be to identify for which smoothing problems this algorithm is the best appropriate and how it maybe combined with existing smoothing algorithms, especially, the one from the family of geometric element transformation methods. \\
 
\bibliographystyle{plain} 
\bibliography{bibfile}

\appendix 
\section{Proof of Lemma~\ref{l:similar}}\label{ap:proof}
Let $\Delta = (x_0,x_1,x_2) \in \left(\mathbb{R}^2\right)^3$ and $\Delta'=(y_0,y_1,y_2) \in \left(\mathbb{R}^2\right)^3$. Assume that there exist an orthogonal matrix $A$, a vector $\sigma \in \mathbb{R}^2$ and a scalar $a \in \mathbb{R}$ such that $y_i = a(Ax_i +\sigma)$ for $i=0,1,2$. Let $c$ be the centroid of $\Delta$, we compute then the centroid of $\Delta'$ as $c' = a(Ac + \sigma)$. As transformation we use for simplicity $\theta$ given by \ref{e.firsttransformation}. We have then for $\theta(\Delta') = (y_0^{(1)},y_1^{(1)},y_2^{(1)})$ that for $i\in \mathbb{Z}$
\begin{align}\label{eq:proof_similar}
y_i^{(1)} &= r_i(y_i - c') + c' \nonumber\\
&=\left\|a(Ax_{i-1} +\sigma) - a(Ac+\sigma)\right\|\left\|a(Ax_i +\sigma) - a(Ac+\sigma)\right\|^{-1}(y_i - c') + c'\nonumber\\
&\;\mbox{norm is absolutely homogeneous}\nonumber\\
&= \left\|A(x_{i-1} - c)\right\|\left\|A(x_i-c)\right\|^{-1}aA(x_i -c) + a(Ac +\sigma) \; \nonumber\\
&\mbox{orthogonal matrices preserve vector lengths}\nonumber\\
&= a A \left(\left\|x_{i-1}-c\right\|\left\|x_i -c\right\|^{-1} (x_i -c) + c\right) + a\sigma \nonumber\\
& = a (A x_i^{(1)}+ \sigma).
\end{align} 
This proves $\theta(\Delta') = \theta(A\Delta + \sigma) = A\theta(\Delta) + \sigma$, i.e. the transformation $\theta$ commutes with isometries in the plane.  
Let us now consider $\theta$ given by Equation~\ref{e.transformation} by $y_i^{(1)} = r_i(y_i-c') + 2c' -c'_{new}$. One easily computes that it also computes with isometries using Equations~\ref{eq:proof_similar}:
\begin{align*}
y_i^{(1)} &= a(Ax_i^{(1)} + \sigma) + c' -c'_{new}\\
&=a(Ax_i^{(1)} + \sigma) + Ac+\sigma - Ac_{new} -\sigma\quad\mbox{using}\; \theta(\Delta') = A\theta(\Delta)+\sigma\\
&=a(A(x_i^{1}+c -c_{new})+\sigma).
\end{align*}
This finishes the proof.
\end{document}